\documentclass{amsart}
\usepackage{amsthm}
\usepackage{amssymb, latexsym}
\usepackage{enumerate}
\usepackage{mathtools,color}
\usepackage{bbm}

\newcommand\supp{\operatorname{supp}}
\newcommand\card{\operatorname{card}}

\newtheorem{prop}{Proposition}
\newtheorem{thm}[prop]{Theorem}
\newtheorem{rem}{Remark}
\newtheorem{cor}{Corollary}
\newtheorem{lem}[prop]{Lemma}
\newtheorem{res}{Result}
\begin{document}

\title[Barely supercritical Schr\"odinger equations]
{On Jensen-type inequalities for nonsmooth radial scattering solutions of a loglog energy-supercritical Schr\"odinger equation
}

\author{Tristan Roy}
%\thanks{$^{*}$ \underline{Note}: this is an update of the published work. Minor corrections + a technical error fixed in Section \ref{Sec:Appendix}. See also
%Remark \ref{Rem:LocalWell} and footnote just above Proposition \ref{Prop:LocalWell} of \cite{triroyfoc}  }

\address{Graduate School of Mathematics. Nagoya University, Japan, \& Department of Mathematics, American University of Beirut, Lebanon}
\email{tr14@aub.edu.lb}

\begin{abstract}
We prove scattering of solutions of the loglog energy-supercritical Schr\"odinger equation
$ i \partial_{t} u + \triangle u = |u|^{\frac{4}{n-2}} u g(|u|) $ with
$g(|u|) := \log^{\gamma} {( \log{(10+|u|^{2})} )} $, $0 < \gamma < \gamma_{n}$, $n \in \{3,4,5 \}$, and with radial data
$u(0) := u_{0} \in \tilde{H}^{k}:= \dot{H}^{k} (\mathbb{R}^{n}) \cap \dot{H}^{1} (\mathbb{R}^{n})$, where $ \frac{n}{2} \geq k > 1$
(resp. $\frac{4}{3} > k > 1$ ) if $n \in \{3,4\}$ (resp. $n=5$). The proof uses concentration techniques (see e.g \cite{bourg,taorad}) to prove a long-time Strichartz-type estimate on an arbitrarily long time interval $J$ depending on an \textit{a priori} bound of some norms of the solution, combined with an induction on time of the Strichartz estimates in order to bound these norms \textit{a posteriori} (see e.g \cite{triroywave,triroyschrod}). We also revisit the scattering theory of solutions with radial data
in $\tilde{H}^{k}$, $k > \frac{n}{2}$, and $n \in \{3,4 \}$: more precisely, we prove scattering for a larger range of $\gamma$ s than in \cite{triroyschrod}. In order to control the barely supercritical nonlinearity for nonsmooth solutions, i.e. solutions with data in $\tilde{H}^{k}$, $k \leq \frac{n}{2}$, we prove some Jensen-type inequalities.

 %(see e.g \cite{gross}) .
\end{abstract}

\maketitle

\section{Introduction}

We shall study the solutions of the following Schr\"odinger equation in dimension $n$, $n \in \{3,4,5 \}$:

\begin{equation}
\begin{array}{ll}
i \partial_{t} u + \triangle u & = |u|^{\frac{4}{n-2}} u g(|u|)
\end{array}
\label{Eqn:BarelySchrod}
\end{equation}
with $g(|u|):= \log^{\gamma}{( \log{ (10 + |u|^{2}) } ) }$, $\gamma > 0 $. This equation has many connections with the following power-type Schr\"odinger equation, $p>1$

\begin{equation}
\begin{array}{ll}
i \partial_{t} v + \triangle v & = |v|^{p-1} v
\end{array}
\label{Eqn:Schrodpowerp}
\end{equation}
(\ref{Eqn:Schrodpowerp}) has a natural scaling: if $v$ is a solution of (\ref{Eqn:Schrodpowerp}) with data $v(0):=v_{0}$ and if $\lambda \in
\mathbb{R}$ is a parameter then $v_{\lambda}(t,x) := \frac{1}{\lambda^{\frac{2}{p-1}}} v \left( \frac{t}{\lambda^{2}}, \frac{x}{\lambda}
\right)$ is also a solution of (\ref{Eqn:Schrodpowerp}) but with data $v_{\lambda}(0,x):= \frac{1}{\lambda^{\frac{2}{p-1}}} v_{0} \left(
\frac{x}{\lambda} \right)$. If $s_{p} := \frac{n}{2}- \frac{2}{p-1}$ then the $\dot{H}^{s_{p}}$ norm of the initial data is invariant under the
scaling: this is why (\ref{Eqn:Schrodpowerp}) is said to be $\dot{H}^{s_{p}}$- critical. If $p=1 + \frac{4}{n-2}$ then (\ref{Eqn:Schrodpowerp})
is $\dot{H}^{1}-$ (or energy-) critical. The energy-critical Schr\"odinger equation
\begin{equation}
\begin{array}{ll}
i \partial_{t} u + \triangle u & = |u|^{\frac{4}{n-2}} u
\end{array}
\label{Eqn:EnergyCrit}
\end{equation}
has received a great deal of attention. Cazenave and Weissler \cite{cazweiss} proved the local well-posedness of (\ref{Eqn:EnergyCrit}): given
any $u(0)$ such that $\| u(0) \|_{\dot{H}^{1}} < \infty$ there exists, for some positive $t_{0}$ close to zero, a unique $ u \in \mathcal{C} ( [0,t_{0}],
\dot{H}^{1} ) \cap L_{t}^{ \frac{2(n+2)}{n-2}} L_{x}^{\frac{2(n+2)}{n-2}} ( [0,t_{0}] )$ satisfying (\ref{Eqn:EnergyCrit}) in the sense of
distributions, hence

\begin{equation}
\begin{array}{ll}
u(t) & = e^{it \triangle} u(0) - i \int_{0}^{t} e^{i(t-t^{'}) \triangle} \left[  |u(t')|^{\frac{4}{n-2}} u(t') \right] \, dt^{'} \cdot
\end{array}
\label{Eqn:DistribSchrod}
\end{equation}
The long-time behavior of radial solutions of (\ref{Eqn:EnergyCrit}) has been studied by several authors. Bourgain \cite{bourg} proved global
well-posedness (i.e global existence) and scattering (i.e linear asymptotic behavior) of the solutions in the class $\mathcal{C} \left( \mathbb{R}, \dot{H}^{1} \right) \cap L_{t}^{\frac{2(n+2)}{n-2}} L_{x}^{\frac{2(n+2)}{n-2}} (\mathbb{R})$ in dimension $n \in \{3,4\}$. He also proved this fact that for smoother solutions. Another proof was given by Grillakis \cite{grill} in dimension $n=3$. The result in the class mentioned above was extended to
higher dimensions (i.e $n \geq 5$) by Tao \cite{taorad}.\\
\\
%The radial assumption for $n=3$ was removed by Colliander-Keel-Staffilani-Takaoka-Tao
%cite{collkeelstafftaktao}. This result was extended to $n=4$ by Rickman-Visan \cite{rickmanvisan} and to $n \geq 5$ by Visan \cite{visan}.
If $p > 1+ \frac{4}{n-2}$ then $s_{p} > 1$ and we are in the energy supercritical regime. Since for all $\epsilon > 0$ there exists $c_{\epsilon} > 0$ such that $ \left| |u|^{\frac{4}{n-2}} u g(|u|) \right| \leq c_{\epsilon} \max{(1, |u|^{\frac{4}{n-2}+ 1 + \epsilon} ) }$ then the nonlinearity
of (\ref{Eqn:BarelySchrod}) is said to be barely supercritical. Barely supercritical equations have been studied extensively in the literature:
see e.g \cite{triroywave, triroybar, triroyschrod, shih, taorad, taowave}.
\\
\\
The global well-posedness and scattering of radial solutions of (\ref{Eqn:BarelySchrod}) lying in $\tilde{H}^{k}$ for $n \in \{3,4\}$ and $k > \frac{n}{2}$ was proved in \cite{triroyschrod} for a range of positive $\gamma$ s. \\
\\
In this paper we are primarily interested in establishing global well-posedness and scattering results of nonsmooth solutions of (\ref{Eqn:BarelySchrod}) for $n \in \{ 3,4,5 \}$. By nonsmooth solutions of (\ref{Eqn:BarelySchrod}) we mean solutions of (\ref{Eqn:BarelySchrod}) lying
in $\tilde{H}^{k}$ with $k \leq  \frac{n}{2}$  $($ The Sobolev embedding says that a function $f$ is continuous if it lies in $\tilde{H}^{k}$, $k > \frac{n}{2}$, but not necessarily if  $ k \leq \frac{n}{2}$. Hence the terminology `` nonsmooth '' $)$. The local well-posedness theory for nonsmooth solutions of (\ref{Eqn:BarelySchrod}) can be formulated as follows:

\begin{prop}
Let $n \in \{3,4,5 \}$. Let $ \frac{n}{2} \geq k > 1$ (resp. $ \frac{4}{3} > k > 1$) if $n \in \{ 3,4 \}$ (resp. $n=5$). Let $M$ be such that $\| u_{0} \|_{\tilde{H}^{k}} \leq M$. Then there exists $\delta:= \delta(M) > 0$  that has the following property: if $T_{l}$ is a number such that
$ T_{l} > 0 $ and

\begin{equation}
\begin{array}{ll}
\| e^{i t \triangle} u_{0}  \|_{L_{t}^{\frac{2(n+2)}{n-2}} L_{x}^{\frac{2(n+2)}{n-2}} ([0,T_{l}]) } & \leq \delta
\end{array}
\label{Eqn:CdtionInit}
\end{equation}
then there exists a unique

\begin{equation}
\begin{array}{l}
u \in \mathcal{C}([0,T_{l}], \tilde{H}^{k}) \cap \mathcal{B} \left( L_{t}^{\frac{2(n+2)}{n-2}} L_{x}^{\frac{2(n+2)}{n-2}} ([0,T_{l}]) ; 2 \delta \right)
\cap L_{t}^{\frac{2(n+2)}{n}} D^{-1} L_{x}^{\frac{2(n+2)}{n}} ([0,T_{l}]) \\
\cap L_{t}^{\frac{2(n+2)}{n}} D^{-k} L_{x}^{\frac{2(n+2)}{n}}
([0,T_{l}])
\end{array}
\nonumber
\end{equation}
such that

\begin{equation}
\begin{array}{l}
u(t)=e^{i t \triangle} u_{0} - i \int_{0}^{t} e^{i(t- t^{'}) \triangle} \left( |u(t^{'})|^{\frac{4}{n-2}} u(t^{'}) g(|u(t^{'})|) \right) \,
dt^{'}
\end{array}
\label{Eqn:DistribSchrodg}
\end{equation}
is satisfied in the sense of distributions. Here $D^{-\alpha} L^{r}:=\dot{H}^{\alpha, r}$ is endowed with the norm $\| f \|_{D^{-\alpha} L^{r}} := \| D^{\alpha} f \|_{L^{r}}$
and $\mathcal{B} \left( L_{t}^{\frac{2{(n+2)}}{n-2}} L_{x}^{\frac{2(n+2)}{n-2}} ([0,T_{l}]);  \bar{r} \right)$ denotes
the closed ball centered at the origin with radius $\bar{r}$ in $L_{t}^{\frac{2(n+2)}{n-2}} L_{x}^{\frac{2(n+2)}{n-2}} ([0,T_{l}])$.
\label{Prop:LocalWell}
\end{prop}

\begin{rem}
$T_{l}$ is called a time of local existence.
\end{rem}

\begin{rem}
Notice that the same proposition as Proposition \ref{Prop:LocalWell} was proved in \cite{triroyschrod} except that the range of $n$ and that of $k$ are replaced with
$ n \in \{ 3,4 \} $ and $ k > \frac{n}{2} $.
\end{rem}

\begin{rem}
In the sequel we call an $\tilde{H}^{k}-$ solution a solution of (\ref{Eqn:BarelySchrod}) that is constructed by Proposition \ref{Prop:LocalWell}.
\end{rem}

The proof of Proposition \ref{Prop:LocalWell} is given in the appendix. This allows (by a standard procedure) to define the notion of maximal time interval of existence $I_{max}$, that is the union of all the open intervals $I$ containing $0$ such that (\ref{Eqn:DistribSchrodg}) holds in the class $ \mathcal{C}( I, \tilde{H}^{k}) \cap L_{t}^{\frac{2(n+2)}{n}} D^{-1} L_{x}^{\frac{2(n+2)}{n}} (I) \cap L_{t}^{\frac{2(n+2)}{n}} D^{-k} L_{x}^{\frac{2(n+2)}{n}} (I) $. The following property holds:

\begin{prop}
If $|I_{max}|< \infty $ then

\begin{equation}
\begin{array}{ll}
\| u \|_{L_{t}^{\frac{2(n+2)}{n-2}} L_{x}^{\frac{2(n+2)}{n-2}} (I_{max})} & = \infty
\end{array}
\end{equation}
\label{Prop:GlobWellPosedCrit}
\end{prop}
Proposition \ref{Prop:GlobWellPosedCrit} is proved in Section \ref{Sec:PropGlob}. With this in mind, global well-posedness follows from an \textit{a priori} bound of the form

\begin{equation}
\begin{array}{l}
\| u \|_{L_{t}^{\frac{2(n+2)}{n-2}} L_{x}^{\frac{2(n+2)}{n-2}} ([-T,T]) } \leq f(T, \| u_{0} \|_{\tilde{H}^{k}})
\end{array}
\end{equation}
for arbitrarily large time $T>0$. In fact we shall prove that the bound does not depend on time $T$: this is the preliminary step to prove
scattering.\\
\\
In this paper we also revisit the asymptotic behavior of radial $\tilde{H}^{k}-$ solutions of (\ref{Eqn:BarelySchrod}) for $n \in \{3,4\}$ with $k > \frac{n}{2}$. In particular, we prove global well-posedness and scattering  of radial $\tilde{H}^{k} -$ solutions of (\ref{Eqn:BarelySchrod}) for a larger range of $\gamma$ s than in \cite{triroyschrod}.\\
\\
The main result of this paper is:

\begin{thm}
Let $n \in \{3,4,5 \}$. Let $I_n$ defined as follows: if $n \in \{3,4\}$ then $I_n := (1,\infty)$ and if
$n=5$ then  $I_n := \left( 1,\frac{4}{3} \right)$. Let $\gamma_n$ be defined as follows: $\gamma_3 = \frac{1}{2744} $, $\gamma_4 = \frac{1}{1600}$, and
$\gamma_5 = \frac{1}{3380}$. \\
The $\tilde{H}^{k}-$ solution of (\ref{Eqn:BarelySchrod}) with radial data $u(0):=u_{0} \in \tilde{H}^{k}$, $k \in I_n$, and
$0 < \gamma < \gamma_{n}$, exists for all time $T$. Moreover there exists a scattering state $u_{0,+} \in \tilde{H}^{k} $ such that

\begin{equation}
\begin{array}{ll}
\lim \limits_{t \rightarrow \infty} \| u(t) - e^{it \triangle} u_{0,+}  \|_{\tilde{H}^{k}} & = 0
\end{array}
\end{equation}
and there exists $C$ depending only on $\| u_{0} \|_{\tilde{H}^{k}}$ such that

\begin{equation}
\begin{array}{ll}
\| u \|_{L_{t}^{\frac{2(n+2)}{n-2}} L_{x}^{\frac{2(n+2)}{n-2}} (\mathbb{R}) } & \leq C ( \| u_{0} \|_{\tilde{H}^{k}} )
\end{array}
\end{equation}

\label{thm:main}
\end{thm}

\begin{rem}
By symmetry $($ i.e if $t \rightarrow u(t,x)$ is a solution of (\ref{Eqn:BarelySchrod}) then
$t \rightarrow \bar{u}(-t,x)$ is a solution of (\ref{Eqn:BarelySchrod}) $)$ there exists $u_{0,-} \in \tilde{H}^{k}$ such that \\
$\lim  \limits_{t \rightarrow -\infty}  \left\| u(t) - e^{i t \triangle}  u_{0,-} \right\|_{\tilde{H}^{k}} =0$.
\end{rem}

\begin{rem}
If $n \in \{ 3,4 \}$ then this theorem and the Sobolev embeddings of $\tilde{H}^{p}$ into $\mathcal{C}^{m}$
(space of functions such that the derivatives of order smaller or equal to $m$ exist and are continuous) for
$p$ and $m$ integers properly chosen imply global results regarding
the regularity of the solutions. For example, the following result holds: if the data is
smooth and radial with enough decay at infinity to be in $\tilde{H}^{k}$ for a $k > \frac{n}{2}$ then for all
time we have a finite bound of the $L^{\infty}$ norm of the solution of (\ref{Eqn:BarelySchrod}). The following
result also holds: if the data is Schwartz and radial then for all time the solution is infinitely differentiable.
\end{rem}

\begin{rem}
If $n \in \{3,4 \}$ and $k > \frac{n}{2}$ global well-posedness and scattering for radial $\tilde{H}^{k} -$ solutions of (\ref{Eqn:BarelySchrod})
were already proved in \cite{triroyschrod} for $ 0 < \gamma < \frac{1}{5772} $ if $n=3$ and for $ 0 < \gamma < \frac{1}{8024}$ if $n=4$. Hence we extend our previous result by covering the range $1 < k \leq \frac{n}{2}$ for $n \in \{3,4\}$ and the range $1 < k < \frac{4}{3}$ for $n=5$. We also prove global well-posedness and scattering with radial data in $\tilde{H}^{k}$, $k > \frac{n}{2}$, for a larger range of $\gamma$ s.
\end{rem}

We set up some notation and recall some estimates. \\
\\
Unless otherwise specified, we let $p'$ be the conjugate of a positive number $p$, i.e $\frac{1}{p} + \frac{1}{p'}=1$.
We write $a \ll b$ (resp. $a \ll_{\alpha} b$) if there exists a positive constant $c \ll 1$(resp. $c:=c(\alpha) \ll 1 $) such that $a \leq c b$, $a \gg b$  (resp. $a \gg_{\alpha} b$) if there exists a positive
constant $C \gg 1$ (resp. $C:= C(\alpha) \gg 1$) such that $a \geq C b $, and $a \approx b$ (resp. $a \approx_{\alpha} b $) if there exists a positive
constant $C$ (resp. $C:=C(\alpha)$) such that $ \frac{1}{C} b \leq  a \leq C b$. If $c := c(\alpha)$ (resp. $C:=C(\alpha)$) but $\alpha$ is not an
important variable ( in the sense that it does not play an important role in the main argument) then for the sake of clarity we forget the dependence
on $ \alpha $ and we write $ a \ll b $ ( resp. $ a \gg b $ ) instead of $ a  \ll_{\alpha} b$ (resp. $a \gg_{\alpha} b$)
$($ e.g., the reader can check
that $C_{1}$ in Proposition \ref{prop:BoundLong} depends on the energy $E$ (see page~\pageref{Eqn:BoundLong}).  Since this dependence is not important, we do not take it into account $)$. The notation above naturally extend to $a \ll_{\alpha_{1},...,\alpha_{m}} b$ by letting the constants depending on $\alpha_{1}$,..., $\alpha_{m}$. If $x \in \mathbb{R} $ then $x \pm := x \pm \epsilon$ for $ 0 < \epsilon \ll 1$ and $\langle x \rangle:= (x^{2} +1)^{\frac{1}{2}}$. \\
\\
Let $f$ be a function depending on space. Let $u$ be a function depending on space and time. Unless
otherwise specified, for sake of simplicity, we do not mention in the sequel the spaces to which $f$ and  $u$
belong in the estimates: this exercise is left to the reader. Let $f_{h}$ denote the function defined by
$x \rightarrow f_{h}(x):=f(x-h)$. The pointwise dispersive estimate is $\| e^{i t \triangle} f  \|_{L^{\infty}} \lesssim \frac{1}{|t|^{\frac{n}{2}}} \| f \|_{L^{1}}$. Interpolating with $\| e^{it \triangle} f \|_{L^{2}} = \| f \|_{L^{2}}$ we have the well-known generalized pointwise dispersive estimate:

\begin{equation}
\begin{array}{ll}
\| e^{it \triangle} f \|_{L^{p}} & \lesssim \frac{1}{|t|^{n \left( \frac{1}{2} - \frac{1}{p} \right)}} \| f
\|_{L^{p'}},
\end{array}
\label{Eqn:DispIneq}
\end{equation}
with $2 \leq p \leq \infty$. \\
Let $r > 1 $ and let $m$ be a positive number such that $m < \frac{n}{r}$. We denote by $m_{r}^{*}$ the number that satisfies
$\frac{1}{m_{r}^{*}} = \frac{1}{r} - \frac{m}{n}$. Let $\bar{k}$ be a constant such that
$1 < \bar{k} < \min \left( \frac{n}{2}, k \right)$. We recall the Sobolev inequalities:

\begin{equation}
\begin{array}{ll}
\| f \|_{L^{\bar{k}_{2}^{*}}} & \lesssim  \| f \|_{\tilde{H}^{k}}, \; \text{and} \\
\| f \|_{L^{m_{r}^{*}}} & \lesssim  \| D^{m} f \|_{L^{r}} \cdot
\end{array}
\label{Eqn:SobolevIneq1}
\end{equation}
We also have

\begin{equation}
\begin{array}{l}
k > \frac{n}{2}: \;  \| f \|_{L^{\infty}} \lesssim \| f \|_{\tilde{H}^{k}} \cdot
\end{array}
\label{Eqn:Sobklarge}
\end{equation}
Let $(\bar{Q},\bar{R})$ be the following

\begin{equation}
\left( \bar{Q},\bar{R} \right) :=
\left\{
\begin{array}{l}
(4+, \infty-) \; \text{if} \; n=3 \\
(2+, \infty-) \; \text{if} \; n=4 \\
(2+, 10-) \; \text{if} \; n=5
\end{array}
\right.
\nonumber
\end{equation}
Let $(\breve{Q}, \breve{R})$ be the following

\begin{equation}
\left( \breve{Q}, \breve{R} \right) :=
\left\{
\begin{array}{l}
(1,2) \; \text{if} \; n \in \{3,4\} \\
\left( \frac{3}{2},\frac{30}{19} \right) \; \text{if} \; n=5
\end{array}
\right.
\nonumber
\end{equation}

Let $J$ be an interval. Let $X(J,u)$ and $Y(J,u)$ denote the following
\begin{equation}
\begin{array}{ll}
X(J,u) := &  \| u \|_{L_{t}^{\infty} \tilde{H}^{k} (J)}  + \| D u \|_{L_{t}^{2} L_{x}^{1_{2}^{*}}(J)} \\
& + \| D^{k} u \|_{L_{t}^{2} L_{x}^{1_{2}^{*}}(J)}, \; \text{and}
\end{array}
\nonumber
\end{equation}

\begin{equation}
\begin{array}{ll}
Y(J,u) := &  \| u \|_{L_{t}^{\infty} \tilde{H}^{k}(J)}  + \| D u \|_{L_{t}^{\frac{2(n+2)}{n}} L_{x}^{\frac{2(n+2)}{n}} (J) } \\
& + \| D^{k} u \|_{L_{t}^{\frac{2(n+2)}{n}} L_{x}^{\frac{2(n+2)}{n}} (J) } \cdot
\end{array}
\nonumber
\end{equation}
We recall the two propositions:

\begin{prop}{\cite{triroyschrod}}
Let $ 0  \leq \alpha < 1$, $k'$ and $\beta$ be integers such that $k' \geq 2$ and $\beta > k'-1 $, $(r ,r_{2}) \in (1,\infty)^{2}$,
$(r_{1},r_{3}) \in (1, \infty]^{2}$ be such that $\frac{1}{r}=
\frac{\beta}{r_{1}} + \frac{1}{r_{2}} +\frac{1}{r_{3}}$. Let $F: \mathbb{R}^{+} \rightarrow \mathbb{R}$ be a $C^{k'}$- function and let $G:=\mathbb{R}^{2} \rightarrow \mathbb{R}^{2}$ be a $C^{k'}$- function such that

\begin{equation}
\begin{array}{l}
F^{[i]}(x) =O \left( \frac{F(x)}{x^{i}} \right), \; \tau \in [0,1]: \;
\left| F \left( |\tau x + (1-\tau)y|^{2} \right) \right|
\lesssim
\left| F(|x|^{2}) \right| + \left| F(|y|^{2}) \right|,
\end{array}
\label{Eqn:Cdtionf}
\end{equation}
and

\begin{equation}
G^{[i]}(x,\bar{x})  = O (|x|^{\beta + 1 -i})
\label{Eqn:CdtionG}
\end{equation}
for $ 0 \leq i \leq k'$. Then

\begin{equation}
\begin{array}{ll}
\left\| D^{ k' -1 + \alpha} ( G(f,\bar{f}) F(|f|^{2}) \right\|_{L^{r}} & \lesssim \| f \|^{\beta}_{L^{r_{1}}} \| D^{k' -1  + \alpha} f \|_{L^{r_{2}}}
\| F(|f|^{2}) \|_{L^{r_{3}}}
\end{array}
\label{Eqn:EstToProveFrac}
\end{equation}
Here $F^{[i]}$ and $G^{[i]}$ denote the $i^{th}-$ derivatives of $F$ and $G$, respectively.
\label{Prop:NonlinFracSmooth}
\end{prop}

\begin{prop}{\cite{triroyschrod}}
Let $(\lambda_1,\lambda_2) \in \mathbb{N}^{2}$ be such that $\lambda_1 + \lambda_2 = \frac{n+2}{n-2}$. Let $J$ be an interval. Let $k > \frac{n}{2}$. Then there exists $\bar{C} >0$ such that
\begin{equation}
\begin{array}{ll}
\left\| D^{k}(u^{\lambda_1} \bar{u}^{\lambda_2} g(|u|) ) \right\|_{L_{t}^{\frac{2(n+2)}{n+4}} L_{x}^{\frac{2(n+2)}{n+4}}(J)} & \lesssim
\| u \|^{\frac{4}{n-2}}_{L_{t}^{\frac{2(n+2)}{n-2}} L_{x}^{\frac{2(n+2)}{n-2}}(J)} \langle  Y(J,u) \rangle^{\bar{C}}
\end{array}
\label{Eqn:EstHighReg}
\end{equation}
\label{Prop:EstHighReg}
\end{prop}
We say that $(q,r)$ is admissible if $q > 2+$ and $\frac{1}{q} + \frac{n}{2r} = \frac{n}{4}$. Let
$(q_{1},r_{1})$ and $(q_{2},r_{2})$ be two bipoints that are admissible. Let $t_{0} \in J$. If $u$ is a solution of
$i \partial_{t} u + \triangle u = G$ on $J$ then the Strichartz estimates (see e.g \cite{keeltao}) yield

\begin{equation}
\begin{array}{l}
\| u \|_{L_{t}^{q_{1}} L_{x}^{r_{1}} (J)} \lesssim  \| u(t_0) \|_{L^{2}} + \| G \|_{L_{t}^{q'_{2}}
L_{x}^{r'_{2}}(J)} \cdot
\end{array}
\label{Eqn:Strich}
\end{equation}
We write

\begin{equation}
\begin{array}{ll}
u(t)= u_{l,t_{0}}(t) + u_{nl,t_{0}}(t),
\end{array}
\end{equation}
with $u_{l,t_{0}}$ denoting the linear part starting from $t_{0}$, i.e

\begin{equation}
\begin{array}{ll}
u_{l,t_{0}}(t) & : = e^{i(t-t_{0}) \triangle} u(t_{0}),
\end{array}
\end{equation}
and $u_{nl,t_{0}}$ denoting the nonlinear part starting from $t_{0}$, i.e

\begin{equation}
\begin{array}{ll}
u_{nl,t_{0}}(t) & := - i \int_{t_{0}}^{t} e^{i(t-s) \triangle} G(s) \,ds \cdot
\end{array}
\end{equation}
If $u$ is a solution of (\ref{Eqn:BarelySchrod}) on $J$ such that $u(t) \in \tilde{H}^{k}$, $t \in J$, then it has a finite energy

\begin{equation}
\begin{array}{ll}
E(u(t))  & := \frac{1}{2} \int_{\mathbb{R}^{n}} |\nabla u (t,x)|^{2} +  \int_{\mathbb{R}^{n}} F(u,\bar{u})(t,x) \, dx,
\end{array}
\label{Eqn:EnergyBarely}
\end{equation}
with

\begin{equation}
\begin{array}{ll}
F(z,\bar{z}) & := \int_{0}^{|z|} t^{\frac{n+2}{n-2}} g(t) \, dt \cdot
\end{array}
\end{equation}
Indeed

\begin{equation}
\begin{array}{ll}
\left| \int_{\mathbb{R}^{n}} F(u,\bar{u})(t,x) \, dx \right| & \lesssim
\| u(t) \|^{\bar{k}_2^{*}}_{L_{\bar{k}_2^{*}}} +  \| u(t) \|^{1_2^{*}}_{L_{1_2^{*}}} \\
&  \lesssim \langle \| u(t) \|_{\tilde{H}^{k}} \rangle^{\bar{k}_2^{*}}:
\end{array}
\nonumber
\end{equation}
this follows from a simple integration by parts

\begin{equation}
\begin{array}{ll}
F(z,\bar{z}) & \sim  |z|^{1_{2}^{*}} g(|z|),
\end{array}
\label{Eqn:EquivF}
\end{equation}
combined with $g(|f|) \lesssim 1 + |f|^{\bar{k}_{2}^{*} -1_{2}^{*}}$ and (\ref{Eqn:SobolevIneq1}). A simple computation shows that the energy is conserved, or, in other words, that $E(u(t))=E(u(0)) = E$ $($ More precisely, if $n \in \{ 3,4 \} $ (resp. $n=5$) then the identity (resp. a similar identity) holds for smooth
solutions (i.e solutions lying in Sobolev spaces with large exponents of (\ref{Eqn:BarelySchrod}) (resp. (\ref{Eqn:BarelySchrod}) ``smoothed'' by
smoothing the nonlinearity at the origin)). Then $E(u(t)) = E(u(0))$ holds for an $\tilde{H}^{k}-$ solution by a standard approximation argument with smooth solutions. $)$. Let $\chi$ be a smooth, radial function supported on $|x| \leq 2$ such that $\chi(x)=1$ if $|x| \leq 1$. If $x_{0} \in \mathbb{R}^{n}$, $R>0$ and $u$ is an $\tilde{H}^{k}-$ solution of (\ref{Eqn:BarelySchrod}) then we define the mass within the ball $B(x_{0},R)$

\begin{equation}
\begin{array}{ll}
\textrm{Mass} \left( B(x_{0},R),u(t)  \right) & :=  \left( \int_{\mathbb{R}^{n}} \chi \left( \frac{x-x_{0}}{R} \right)  |u (t,x) |^{2} \, dx \right)^{\frac{1}{2}}
\end{array}
\label{Eqn:MassControl0}
\end{equation}
Recall (see e.g \cite{taorad}) that

\begin{equation}
\begin{array}{ll}
\textrm{Mass} \left( B(x_{0},R), u(t)  \right) & \lesssim  R \,  \| \nabla u(t) \|_{L^{2}}
\end{array}
\label{Eqn:MassControl}
\end{equation}
and that its derivative satisfies $($ It is also well-known that if $u$ is a linear $\tilde{H}^{k}-$ solution
(that is a solution of the linear Schr\"odinger equation with data in $\tilde{H}^{k}$), then (\ref{Eqn:UpBdDerivM}) also holds. $)$

\begin{equation}
\begin{array}{ll}
\left| \partial_{t} \textrm{Mass}(u(t),B(x_{0},R)) \right| & \lesssim \frac{ \| \nabla u(t) \|_{L^{2}}}{R} \cdot
\end{array}
\label{Eqn:UpBdDerivM}
\end{equation}
We recall the following proposition:

\begin{prop}{\cite{triroyschrod}}
Let $u$ be a solution of (\ref{Eqn:BarelySchrod}) with data
$u_0 \in \tilde{H}^{k}$, $k > \frac{n}{2}$. Assume that $u$ exists globally in time and that
$\| u \|_{L_{t}^{\frac{2(n+2)}{n-2}} L_{x}^{\frac{2(n+2)}{n-2}} (\mathbb{R})} < \infty$. Then
$Y(\mathbb{R},u) < \infty$.
\label{Prop:PersReg}
\end{prop}
We now explain the main ideas of this paper.\\
\\
In Section \ref{Sec:Thmmain} we prove the main result of this paper, i.e Theorem \ref{thm:main}. The
proof relies upon the following bound of $\| u \|_{L_{t}^{\bar{Q}} L_{x}^{\bar{R}} }$ on an arbitrarily long time interval:

\begin{prop}
Let $u$ be a radial $\tilde{H}^{k} -$ solution of (\ref{Eqn:BarelySchrod}) on an interval $J:=[a,b]$. There exists a constant $C_{1} \gg  1$ such that
if $X(J,u) \leq M$ for some $M \gg 1$, then

\begin{equation}
\begin{array}{ll}
\| u \|^{\bar{Q}}_{ L_{t}^{\bar{Q}} L_{x}^{\bar{R}} (J)} & \leq  C_{1}^{C_{1}
g^{b_{n}+}(M)}
\end{array}
\label{Eqn:BoundLong}
\end{equation}
with $b_{n}$ such that

\begin{equation}
\begin{array}{l}
b_{n} := \left\{
\begin{array}{l}
2744, \, \,   n=3 \\
1600 , \, n=4 \\
3380, \, n=5 \cdot
\end{array}
\right.
\end{array}
\end{equation}
\label{prop:BoundLong}
\end{prop}
The proof of this proposition is given in Section \ref{Sec:Boundlong}. We aim at establishing bounds of norms of the solution that do not depend on time at a higher regularity (i.e $\tilde{H}^{k}$) than the energy (i.e $\dot{H}^{1}$) on an arbitrarily long-time interval. To this end we proceed in two steps (see e.g \cite{triroyschrod,triroywave}). First we establish a Strichartz-type estimate on an arbitrarily long-time interval that depends on an \textit{a priori} bound of these norms: see Section \ref{Sec:Boundlong}. Then we find an \textit{a posteriori} bound by combining this estimate with a local induction on time of the Strichartz estimates: see Section \ref{Sec:Thmmain}.
In the first step, we use the techniques of concentration to establish the Strichartz-type estimate by modifying closely an argument in \cite{taorad}. Roughly speaking, we divide the long-time interval into subintervals where the Strichartz-type
norm is small but not so small. Our goal then boils down to find an explicit bound of the number of subintervals by using local estimates ont these
subintervals and a Morawetz-type inequality. A key element in the process of establishing this bound is to use the slow increase of the
function $g$ by making the estimates involving the expressions where $g$ appears depend on $g$ evaluated at the \textit{a priori} bound, and not only on the
\textit{a priori} bound. The function $g$ appears whenever one has to control the nonlinearity on these subintervals. The nonlinearity is
controlled by using a fractional Leibniz rule and the smallness of the Strichartz-type norm on these subintervals. In \cite{triroyschrod}, we used
extensively the boundedness of the solutions (in other words the Sobolev embedding $\| f \|_{L^{\infty}} \lesssim \| f \|_{\tilde{H}^{k}}$), using to
our advantage $k > \frac{n}{2}$ for $n \in \{3,4\}$, in order to derive estimates that depend on $g$ evaluated at the \textit{a priori} bound. In this paper, in order to deal with nonsmooth solutions, we prove some inequalities  (the so-called Jensen-type inequalities) that are substitutes for the Sobolev
embedding and we implement them in order to prove estimates that satisfy the same property as that stated above.\\
We also use this opportunity to revisit the asymptotic behavior of radial $\tilde{H}^{k}-$ solutions of (\ref{Eqn:BarelySchrod}) for $k > \frac{n}{2}$ and $n \in \{3,4\}$. We prove global well-posedness and scattering  of radial $\tilde{H}^{k}-$ solutions of (\ref{Eqn:BarelySchrod}) for
a larger range of $\gamma$ s than in \cite{triroyschrod} by optimizing the algorithm and the value of the parameters (such as the value of $\bar{Q}$ and $\bar{R}$) in Proposition \ref{prop:BoundLong} and its proof.

\section{Jensen-type inequalities}
\label{Sec:Jensen}

We prove the following Jensen-type inequalities:

\begin{prop}

Let $J$ be an interval. Let $\beta > 0$. Denote by $\mathcal{P}$ the following set

\begin{equation}
\mathcal{P} := \left\{ (x,y):
\begin{array}{l}
\frac{1}{x} + \frac{n}{2y} = \frac{n-2}{4} \\
n=5: \; x \geq 2, \; n=4: \; x > 2, \; n=3: \; x > 4
\end{array}
\right\} \cdot
\nonumber
\end{equation}
Let $(q,r) \in \mathcal{P}$. Let $\frac{1}{\bar{r}} = \frac{n-2\bar{k}}{r(n-2)}$. Let $(q',r')$  be such that

\begin{equation}
\left\{
\begin{array}{l}
\text{if} \; q \neq \infty: \; \left( \frac{1}{q'}, \frac{1}{r'} \right) :=
\left( \frac{n-2\bar{k}}{4} - \frac{n}{2 \bar{r}}, \frac{1}{\bar{r}} + \frac{\bar{k}}{n} \right) \\
\text{if} \; q = \infty: \; \left( \frac{1}{q'}, \frac{1}{r'} \right) :=  \left( 0,\frac{1}{2} \right) \cdot
\end{array}
\right.
\label{Eqn:Cdtionq}
\end{equation}
Assume that there exist $0< P \lesssim 1$ and $Q$ such that
$ \| u \|_{ L_{t}^{q} L_{x}^{r}(J)} \leq P $ and $ \| D^{\bar{k}} u \|_{L_{t}^{q'} L_{x}^{r'}(J)} \leq Q$. Then

\begin{equation}
\begin{array}{ll}
\left\| g^{\beta}(|u|) u \right\|_{L_{t}^{q} L_{x}^{r}(J)} \lesssim P \left(  g^{\beta}(Q) + P^{0-} \right)  \cdot
\end{array}
\label{Eqn:Jensen}
\end{equation}
\label{Prop:Jensen}
\end{prop}

\begin{proof}
Let $ 1 \gg \epsilon > 0$ be a fixed constant. Elementary considerations show that there exists $A \approx 1 $ such that if $|x| > A$
then $g^{\beta m}$ is concave with $m \in \{ q,r \}$ and $g^{\beta r}(|x|^{\epsilon}) \geq \frac{1}{10} g^{\beta r}(|x|)$. \\
\\
\begin{itemize}

\item \underline{Case $1$}: $q = \infty$  $($ Hence $r =1_{2}^{*} = \frac{2n}{n-2}$ $)$ \\
\\
Let $ \bar{k}_{2}^{*} -r \gg \epsilon > 0$ be a fixed constant. One has to estimate for all $t \in J$

\begin{equation}
\begin{array}{l}
X_{1} := \int_{|u(t,x)| \leq A} g^{\beta r}(|u(t,x)|) |u(t,x)|^{r} \; dx,\; \text{and} \\
X_{2} := \int_{|u(t,x)| > A} g^{\beta r}(|u(t,x)|) |u(t,x)|^{r} \; dx.
\end{array}
\nonumber
\end{equation}
Clearly $ |X_{1}|  \lesssim P^{r} $. Observe also from (\ref{Eqn:SobolevIneq1}) that

\begin{equation}
\begin{array}{ll}
\| u(t) \|_{L^{r + \epsilon}} & \lesssim \| u(t) \|^{1-\theta}_{L^{\bar{k}_{2}^{*}}} \| u(t) \|^{\theta}_{L^{r}} \\
& \lesssim Q^{1- \theta} \| u(t) \|_{L^{r}}^{\theta},
\end{array}
\nonumber
\end{equation}
with $\theta:= \frac{\frac{1}{_{r+ \epsilon}} - \frac{1}{_{\bar{k}_{2}^{*}}}}{\frac{1}{_{r}} - \frac{1}{_{\bar{k}_{2}^{*}}}} =
\frac{ \left( \bar{k}_{2}^{*} - (r + \epsilon) \right) r }{(r+ \epsilon)\left( \bar{k}_{2}^{*}-r \right)}$. We get from the Jensen inequality
$($ \hspace{0.1cm} more precisely divide into two regions $ R_{1}:= \{ x: \;  |u(t,x)|^{\epsilon} \leq A \}$ and
$R_{2} :=  \{  x: \;  |u(t,x)|^{\epsilon} > A \} $; on $R_{1}$ apply elementary estimates and on
$R_{2}$  use $ \mathbbm{1}_{R_{2}}(x) g^{\beta r}\left( |u(t,x)|^{\epsilon} \right) \lesssim
 g^{\beta r} \left( \mathbbm{1}_{R_{2}}(x)  |u(t,x)|^{\epsilon} \right) $ and apply the Jensen inequality  for $t$ such that $ \| u(t) \|^{r}_{L^{r}} \neq 0 $ with $\mu$ measure defined by $ d \mu :=  |u(t)|^{r} \; dx $ $)$

\begin{equation}
\begin{array}{ll}
|X_{2}| & \lesssim \int g^{\beta r} \left( |u(t,x)|^{\epsilon} \right) |u(t,x)|^{r} \; dx  \\
& \lesssim \int |u(t,x)|^{r} \; dx \;  g^{\beta r} \left(  \frac{ \int |u(t,x)|^{r + \epsilon} \; dx}{\int |u(t,x)|^{r} \; dx} \right) \\
& \lesssim  \int |u(t,x)|^{r} \; dx \; g^{\beta r} \left(
\frac{ Q^{(r+\epsilon)(1- \theta)} }{ \| u(t) \|_{L^{r}}^{r- \theta (r+ \epsilon)}} \right)
\end{array}
\nonumber
\end{equation}
Elementary estimates show that

\begin{equation}
\begin{array}{ll}
|X_{2}| &  \lesssim \| u(t) \|_{L^{r}}^{r} \left(  g^{\beta r}(Q) + g^{\beta r} \left( \frac{1}{\| u(t) \|_{L^{r}}} \right) \right)
\lesssim  P^{r}\left(  g^{\beta r}(Q) + P^{0-} \right)
\end{array}
\nonumber
\end{equation}
Hence (\ref{Eqn:Jensen}) holds. \\
\\
\item \underline{Case $2$}: $q < \infty$ \\
\\
Let $ \bar{k} - 1 \gg \epsilon > 0$ be a fixed constant. One has to estimate

\begin{equation}
\begin{array}{l}
 X_{1} := \int_{J} \left( \int_{|u(t,x)| \leq A} g^{\beta r}(|u(t,x)|) |u(t,x)|^{r} \; dx \right)^{\frac{q}{r}}
 \; dt, \; \text{and} \\
X_{2} := \int_{J} \left( \int_{|u(t,x)| \geq A} g^{\beta r}(|u(t,x)|) |u(t,x)|^{r} \; dx \right)^{\frac{q}{r}} \; dt.
\end{array}
\nonumber
\end{equation}
Clearly $ |X_1| \lesssim  P^{q}$. From (\ref{Eqn:SobolevIneq1}) we get

\begin{equation}
\begin{array}{ll}
\| u \|_{L_{t}^{\frac{q(r+ \epsilon)}{r}} L_{x}^{r+ \epsilon}(J)} & \lesssim  \| u \|^{1 - \theta}_{L_{t}^{q} L_{x}^{r}(J)}
\| u \|^{\theta}_{L_{t}^{q'} L_{x}^{\bar{r}}(J)} \\
& \lesssim \| u \|^{1 - \theta}_{L_{t}^{q} L_{x}^{r}(J)} \| D^{\bar{k}} u \|^{\theta}_{ L_{t}^{q'} L_{x}^{r'} (J)} \\
& \lesssim Q^{\theta} \| u \|^{1- \theta}_{L_{t}^{q} L_{x}^{r}(J)},
\end{array}
\nonumber
\end{equation}
with $ \theta := \frac{\frac{1}{r} - \frac{1}{r+ \epsilon}}{\frac{1}{r} - \frac{1}{\bar{r}}}
=  \frac{\epsilon (n-2)}{_{2(r+ \epsilon)(\bar{k}-1)}} $. Applying twice the Jensen inequality

\begin{equation}
\begin{array}{ll}
|X_{2}| & \lesssim \int_{J} \left( \int  g^{\beta r}(|u(t,x)|^{\epsilon}) |u(t,x)|^{r} \; dx \right)^{\frac{q}{r}} \; dt  \\
& \lesssim
\int_{J} \left( \int |u(t,x)|^{r} \; dx \;  g^{\beta r} \left(  \frac{ \int |u(t,x)|^{r + \epsilon} \; dx}
{\int |u(t,x)|^{r} \; dx} \right) \right)^{\frac{q}{r}} \; dt \\
& \lesssim \int_{J} \left( \int |u(t,x)|^{r} \; dx  \right)^{\frac{q}{r}} \;  g^{\beta q}
\left( \left(  \frac{ \int |u(t,x)|^{r + \epsilon} \; dx}
{\int |u(t,x)|^{r} \; dx} \right)^{\frac{q}{r}} \right) \; dt \\
& \lesssim \| u \|_{L_{t}^{q} L_{x}^{r}(J)}^{q} \;
g^{\beta q} \left(  \frac{  \| u \|^{\frac{q(r+\epsilon)}{r}}_{L_{t}^{\frac{q(r+\epsilon)}{r}} L_{x}^{r+ \epsilon}(J)}}
{\| u \|^{q}_{L_{t}^{q} L_{x}^{r}(J)}} \right) \\
& \lesssim \| u \|^{q}_{L_{t}^{q} L_{x}^{r}(J)} g^{\beta q} \left(
\frac{ Q^{\frac{\theta q (r + \epsilon)}{r}}}
{\| u \|^{q \left( 1 - \frac{(1- \theta)(r + \epsilon)}{r} \right)}_{L_{t}^{q} L_{x}^{r}(J)}} \right)
\end{array}
\nonumber
\end{equation}
Elementary estimates show that

\begin{equation}
\begin{array}{l}
|X_2| \lesssim \| u \|^{q}_{L_{t}^{q} L_{x}^{r}(J)}
\left( g^{\beta q} (Q) +  g^{\beta q} \left( \frac{1}{\| u \|_{L_{t}^{q} L_{x}^{r}(J)}} \right)  \right)
\lesssim P^{q} \left( g^{\beta q} (Q) + P^{0-} \right) \cdot
\end{array}
\nonumber
\end{equation}
Hence (\ref{Eqn:Jensen}) holds.

\end{itemize}
\end{proof}

\section{Consequences}

In this section, we implement the Jensen-type inequalities to prove some results.

\subsection{Fractional Leibniz rule}

We prove the following fractional Leibnitz rule:

\begin{prop}
Let $J$ be an interval. Let $ 0  \leq \alpha \leq 1$, $(q,r) \in \mathcal{P}$, and $(\breve{q},\tilde{q},\breve{r},\tilde{r})$ be such that $\left( \frac{1}{\breve{q}}, \frac{1}{\breve{r}} \right) = \frac{4}{n-2} \left( \frac{1}{q}, \frac{1}{r} \right)
+ \left( \frac{1}{\tilde{q}}, \frac{1}{\tilde{r}} \right)$. Let
$G:=\mathbb{R}^{2} \rightarrow \mathbb{R}^{2}$ be a $C^{2}$- function such that

\begin{equation}
G^{[i]}(x,\bar{x})  = O (|x|^{\frac{4}{n-2}+ 1 -i})
\label{Eqn:CdtionG}
\end{equation}
for $ 0 \leq i \leq 2$. Here $G^{[i]}$ denotes the $i^{th}-$ derivative of $G$. Let $(q',r') $ be a bipoint that satisfies (\ref{Eqn:Cdtionq}).
Assume that there exist $ P \lesssim 1 $ and $Q$ such that $\| u \|_{L_{t}^{q} L_{x}^{r}(J)} \leq P$
and $\| D^{\bar{k}} u \|_{L_{t}^{q'} L_{x}^{r'}(J)} \leq Q$. Then

\begin{equation}
\begin{array}{ll}
\left\| D^{1 + \alpha} ( G(u,\bar{u}) g(|u|) \right\|_{L_{t}^{\breve{q}} L_{x}^{\breve{r}}(J)} & \lesssim
P^{\frac{4}{n-2}-} g(Q)  \| D^{1 + \alpha} u \|_{L_{t}^{\tilde{q}} L_{x}^{\tilde{r}}(J)}
\end{array}
\label{Eqn:EstToProveFrac}
\end{equation}
\label{Prop:FracLeibn}
\end{prop}

\begin{proof}
The proof combines Jensen-type inequalities with similar arguments that are in the proof of the fractional Leibnitz rule
established in \cite{triroyschrod}. \\
Recall the usual product rule for fractional derivatives

\begin{equation}
\begin{array}{ll}
\| D^{\alpha_{1}} (fg) \|_{L^{q}} & \lesssim \| D^{\alpha_{1}} f \|_{L^{q_{1}}} \| g \|_{L^{q_{2}}} + \| f \|_{L^{q_{3}}} \| D^{\alpha_{1}} g
\|_{L^{q_{4}}}
\end{array}
\label{Eqn:FracProd}
\end{equation}
and the usual Leibnitz rule for fractional derivatives :

\begin{equation}
\begin{array}{ll}
\| D^{\alpha_{2}} H(f) \|_{L^{q}} & \lesssim \| \tilde{H}(f) \|_{L^{q_{1}}}  \| D^{\alpha_{2}} f \|_{L^{q_{2}}}
\end{array}
\label{Eqn:DerivComp}
\end{equation}
if $H$ is $C^{1}$ and it satisfies $\left| H^{'}\left( \tau x + (1- \tau) y \right) \right| \lesssim
\tilde{H}(x) + \tilde{H}(y)$ for $\tau \in [0,1]$, $ 0 \leq \alpha_{1} < \infty$, $ 0 <  \alpha_{2}  \leq 1 $, $(q,q_{1},q_{2},q_{3},q_4) \in (1, \infty)^{5}$, $\frac{1}{q}= \frac{1}{q_{1}} + \frac{1}{q_{2}}$, and $\frac{1}{q}= \frac{1}{q_3} + \frac{1}{q_4}$ ( see e.g Christ-Weinstein \cite{christwein}, Taylor
\cite{taylor} and references in \cite{taylor}) $($ Abuse of notation: $\tilde{H}(x)$, $H(x)$, and $H^{'}(x)$ mean
$\tilde{H}(x,\bar{x})$, $H(x,\bar{x})$, and $H^{'}(x,\bar{x})$ respectively \label{foot:notab} $)$. Let $\tilde{g}(v) := \log^{\gamma} \log(10 + v)$ for $v \in \mathbb{R}^{+}$. We have

\begin{equation}
\begin{array}{ll}
\left\| D^{1 +\alpha} ( G(u,\bar{u}) g(|u|) ) \right\|_{L_{t}^{\breve{q}} L_{x}^{\breve{r}}(J)} \\
\approx \left\| D^{\alpha } \nabla ( G(u,\bar{u}) \tilde{g}(|u|^{2})  ) \right\|_{L_{t}^{\breve{q}} L_{x}^{\breve{r}}(J)} \\
\lesssim \left\| D^{\alpha}( \partial_{z} G(u,\bar{u}) \nabla u  \tilde{g}(|u|^{2}) )  \right\|_{L_{t}^{\breve{q}} L_{x}^{\breve{r}}(J)} +
\left\| D^{\alpha} ( \partial_{\bar{z}} G(u,\bar{u}) \overline{ \nabla u}  \tilde{g}(|u|^{2}) ) \right\|_{L_{t}^{\breve{q}} L_{x}^{\breve{r}}(J)} \\
+ \left\| D^{\alpha} \left( \tilde{g}^{'}(|u|^{2})   \Re \left( \bar{u}  \nabla u \right)  G(u,\bar{u})  \right) \right\|_{L_{t}^{\breve{q}}
L_{x}^{\breve{r}} (J)}  \\
\lesssim  A_{1} + A_{2} + A_{3}
\end{array}
\end{equation}
We estimate $A_{1}$. $A_{2}$ is estimated in a similar fashion. Let
$\left( \frac{1}{q_2}, \frac{1}{r_2} \right) = (1-\theta) \left( \frac{1}{q},\frac{1}{r} \right) + \theta
\left( \frac{1}{_{\tilde{q}}}, \frac{1}{_{\tilde{r}}} \right)$ with $\theta= \frac{1}{1+ \alpha}$. Let $(q_1,r_1)$ be such that
$\left( \frac{1}{_{\breve{q}}}, \frac{1}{_{\breve{r}}} \right)= \left( \frac{1}{q_1} + \frac{1}{q_2}, \frac{1}{r_1} + \frac{1}{r_2} \right)$. We
can estimate $A_1$ using (\ref{Eqn:FracProd}) and (\ref{Eqn:DerivComp}). More
precisely

\begin{equation}
\begin{array}{ll}
A_{1} & \lesssim \| D^{\alpha} ( \partial_{z} G(u,\bar{u})  \tilde{g}(|u|^{2}) )  \|_{L_{t}^{q_{1}} L_{x}^{r_{1}}(J)}
\| D u \|_{L_{t}^{q_2} L_{x}^{r_{2}}(J)} +
\| \partial_{z} G(u,\bar{u}) \tilde{g}(|u|^{2}) \|_{L_{t}^{\frac{n-2}{4}q} L_{x}^{\frac{n-2}{4}r}(J)} \\
& \| D^{ 1 + \alpha} u \|_{L_{t}^{\tilde{q}} L_{x}^{\tilde{r}}(J)}
\end{array}
\label{Eqn:EstA1}
\end{equation}
Next we implement the Jensen-type inequalities: see Section \ref{Sec:Jensen}. Let $(q_3,r_3)$ be such that
$\left( \frac{6-n}{(n-2)q},\frac{6-n}{(n-2)r} \right) +  \left( \frac{1}{q_3}, \frac{1}{r_3} \right) = \left( \frac{1}{q_1}, \frac{1}{r_1} \right)$.
From

\begin{equation}
\begin{array}{l}
\left\|  \left( \partial_{z} G(z,\bar{z}) \tilde{g}(|z|^{2}) \right)^{'}(u,\bar{u}) \right\|_{L_{t}^{\frac{n-2}{6-n}q}
L_{x}^{\frac{n-2}{6-n}r}(J)}  \lesssim
\| u  \tilde{g}^{\frac{n-2}{6-n}}(|u|^{2}) \|^{\frac{6-n}{n-2}}_{L_{t}^{q} L_{x}^{r}(J)} \lesssim P^{\frac{6-n}{n-2}-}g(Q), \; \text{and}
\end{array}
\nonumber
\end{equation}

\begin{equation}
\begin{array}{ll}
\| \partial_{z} G(u,\bar{u}) \tilde{g}(|u|^{2}) \|_{L_{t}^{\frac{n-2}{4}q} L_{x}^{\frac{n-2}{4}r}(J)} \lesssim
\| u \tilde{g}^{\frac{n-2}{4}}(|u|^{2}) \|^{\frac{4}{n-2}}_{L_{t}^{q} L_{x}^{r} (J)} \lesssim
P^{\frac{4}{n-2}-} g(Q),
\end{array}
\nonumber
\end{equation}
we get

\begin{equation}
\begin{array}{ll}
A_{1} & \lesssim P^{\frac{6-n}{n-2}-} g(Q)
\| D^{\alpha} u \|_{L_{t}^{q_3} L_{x}^{r_3}(J)} \| D u  \|_{L_{t}^{q_2} L_{x}^{r_2}(J)} \\
& +  P^{\frac{4}{n-2}-} g(Q)   \| D^{1+\alpha} u \|_{L_{t}^{\tilde{q}} L_{x}^{\tilde{r}}(J)} \cdot
\end{array}
\label{Eqn:EstFrac1}
\end{equation}
Notice that $\frac{1}{q_3} = \frac{\theta}{q} + \frac{1- \theta}{\tilde{q}}$ and
$\frac{1}{r_3} = \frac{\theta}{r} + \frac{1- \theta}{\tilde{r}}$. By complex interpolation, we have

\begin{equation}
\begin{array}{ll}
\| D^{\alpha} u \|_{L_{t}^{q_{3}} L_{x}^{r_{3}}(J)} & \lesssim \| u \|^{\theta}_{L_{t}^{q} L_{x}^{r}(J)}
\| D^{ 1 + \alpha}  u \|^{1- \theta}_{L_{t}^{\tilde{q}} L_{x}^{\tilde{r}}(J)}
\end{array}
\label{Eqn:Interp11}
\end{equation}
and

\begin{equation}
\begin{array}{ll}
 \| D u  \|_{ L_{t}^{q_{2}} L_{x}^{r_{2}}(J)} & \lesssim \| u \|^{1- \theta}_{L_{t}^{q} L_{x}^{r}(J)}
 \| D^{ 1 + \alpha} u  \|^{\theta}_{L_{t}^{\tilde{q}} L_{x}^{\tilde{r}}(J)}
\end{array}
\label{Eqn:Interp12}
\end{equation}
Plugging (\ref{Eqn:Interp11}) and (\ref{Eqn:Interp12}) into (\ref{Eqn:EstFrac1}) we get (\ref{Eqn:EstToProveFrac}).

We estimate $A_{3}$.

\begin{equation}
\begin{array}{ll}
A_{3} & \lesssim \sum \limits_{\tilde{u} \in \{ u,\bar{u} \}}
\left[
\begin{array}{l}
\left\| D^{\alpha} \left( \tilde{g}^{'}(|u|^{2})  \tilde{u} G(u,\bar{u}) \right) \right\|_{L_{t}^{q_1} L_{x}^{r_1}(J)}
\| D u \|_{L_{t}^{q_2} L_{x}^{r_2}(J)} \\
+ \| D^{1+ \alpha} u \|_{L_{t}^{\tilde{q}} L_{x}^{\tilde{r}}(J)}
\left\| \tilde{g}^{'}(|u|^{2}) \tilde{u}  G(u,\bar{u}) \right\|_{L_{t}^{\frac{(n-2)q}{4}} L_{x}^{\frac{(n-2)r}{4}}(J)}
\end{array}
\right] \\
& \lesssim A_{3,1} + A_{3,2}
\end{array}
\end{equation}
Hence from elementary pointwise estimates of $\tilde{g}$ (and its derivatives) we see that $A_{3,1}$ (resp. $A_{3,2}$) can be
estimated similarly to the first term (resp. the second term) of the right-hand side of (\ref{Eqn:EstA1}).

\end{proof}

\subsection{Corollary}

We prove the following corollary:

\begin{cor}

Let $\breve{q}$, $\breve{r}$, $q$, $r$, $\tilde{q}$,  $\tilde{r}$, $\Omega$  be such that \\
$\left( \breve{q},\breve{r}, q, r, \tilde{q},\tilde{r}, \Omega \right) :=
\left( \breve{Q},\breve{R}, \bar{Q}, \bar{R}, \infty-, 2+, X \right) $ or \\
$\left( \breve{q},\breve{r}, q, r, \tilde{q},\tilde{r}, \Omega \right) := \left( \frac{2(n+2)}{n+4}(1,1), \frac{2(n+2)}{n-2}(1,1),
\frac{2(n+2)}{n} (1,1),Y \right) $. \\
Let $J$ be an interval such that $\| u \|_{L_{t}^{q} L_{x}^{r}(J)} \leq P \lesssim 1$.
Let $j \in \{1,k \}$. Assume that $1 <  k \leq 2$. Then

\begin{equation}
\begin{array}{ll}
\left\| D^{j} \left( |u|^{\frac{4}{n-2}} u g(|u|) \right) \right\|_{L_{t}^{\breve{q}}  L_{x}^{\breve{r}} (J)}
& \lesssim P^{\frac{4}{n-2}-} g \left( \Omega(J,u) \right) \| D^{j} u \|_{L_{t}^{\tilde{q}} L_{x}^{\tilde{r}} (J)}  \cdot
\end{array}
\label{Eqn:EstDj}
\end{equation}

\label{Cor:Nonlin}
\end{cor}

\begin{proof}
Let $\theta \in [0,1]$ be a constant that is allowed to change from one line to the other one and such that all the
estimate  below are true. \\
We apply Proposition \ref{Prop:FracLeibn}. \\
Let $(q^{''},r^{''})$  be defined as follows:

\begin{equation}
(q^{''},r^{''}) :=
\left\{
\begin{array}{l}
(2,1_{2}^{*}) \; \text{if} \; \Omega=X  \\
\frac{2(n+2)}{n} (1,1) \; \text{if} \; \Omega=Y \cdot
\end{array}
\right.
\nonumber
\end{equation}
Observe that

\begin{equation}
\begin{array}{ll}
\| D^{\bar{k}} u \|_{L_{t}^{q'} L_{x}^{r'}(J)} & \lesssim
\| D^{\bar{k}} u \|^{\theta}_{L_{t}^{q^{''}} L_{x}^{r^{''}} (J) }
\| D^{\bar{k}} u \|^{1- \theta}_{L_{t}^{\infty} L_{x}^{2}(J)},  \\
\| D^{\bar{k}} u \|_{L_{t}^{q^{''}} L_{x}^{r^{''}} (J)} & \lesssim  \| D^{k} u \|^{\theta}_{L_{t}^{q^{''}} L_{x}^{r^{''}} (J) }
\| D u \|^{1- \theta}_{L_{t}^{q^{''}} L_{x}^{r^{''}} (J)}, \; \text{and} \\
\| D^{\bar{k}} u \|_{L_{t}^{\infty} L_{x}^{2} (J)} & \lesssim  \| D^{k} u \|^{\theta}_{L_{t}^{\infty} L_{x}^{2} (J) }
\| D u \|^{1- \theta}_{L_{t}^{\infty} L_{x}^{2} (J)} \cdot
\end{array}
\label{Eqn:Interpkbar}
\end{equation}
This yields (\ref{Eqn:EstDj}).

\end{proof}

\section{Proof of Proposition \ref{Prop:GlobWellPosedCrit}}
\label{Sec:PropGlob}
In this section we prove Proposition \ref{Prop:GlobWellPosedCrit}. \\
\\
Assume that $ \| u \|_{L_{t}^{\frac{2(n+2)}{n-2}} L_{x}^{\frac{2(n+2)}{n-2}} (I_{max})} < \infty$. \\
Let $J:=[0,a]$ be an interval such that $0  \in J$ and $\| u \|_{L_{t}^{\frac{2(n+2)}{n-2}} L_{x}^{\frac{2(n+2)}{n-2}}(J)} \lesssim 1$. By (\ref{Eqn:Strich}) and Corollary \ref{Cor:Nonlin}   we have

\begin{equation}
\begin{array}{ll}
Y(J,u) &  \lesssim  \| u_{0} \|_{\tilde{H}^{k}} +   \| D ( |u|^{\frac{4}{n-2}} u g(|u|) ) \|_{ L_{t}^{\frac{2(n+2)}{n+4}}
L_{x}^{\frac{2(n+2)}{n+4}}(J) } \\
& +  \| D^{k} ( |u|^{\frac{4}{n-2}} u g(|u|) ) \|_{ L_{t}^{\frac{2(n+2)}{n+4}} L_{x}^{\frac{2(n+2)}{n+4}}(J)} \\
& \leq \bar{C} \| u_{0} \|_{\tilde{H}^{k}} + 2 \bar{C} Y(J,u) \| u \|^{\frac{4}{n-2}-}_{L_{t}^{\frac{2(n+2)}{n-2}} L_{x}^{\frac{2(n+2)}{n-2}} (J) }
g(Y(J,u)),
\end{array}
\label{Eqn:QJu}
\end{equation}
where $\bar{C}$ is a fixed, large, and positive constant. \\
Let $0 < \epsilon \ll 1$. We may assume without loss of generality that $\bar{C} \gg \max \left( \| u_0 \|^{100}_{\tilde{H}^{k}}, \frac{1}{\| u_0 \|^{100}_{\tilde{H}^{k}}} \right)$. We divide $I_{max} \cap [0,\infty)$ into subintervals $(I_{j})_{1 \leq j \leq J}$ such that $0 \in I_1$,

\begin{equation}
\begin{array}{ll}
\| u \|_{L_{t}^{\frac{2(n+2)}{n-2}} L_{x}^{\frac{2(n+2)}{n-2}} (I_{j})} & = \frac{\epsilon}{g^{\frac{n-2}{4}+}\left( (2\bar{C})^{j} \| u_{0} \|_{\tilde{H}^{k}} \right) }
\end{array}
\nonumber
\end{equation}
if $1 \leq j <J$ and

\begin{equation}
\begin{array}{ll}
\| u \|_{L_{t}^{\frac{2(n+2)}{n-2}} L_{x}^{\frac{2(n+2)}{n-2}} (I_{J})} & \leq \frac{\epsilon}{g^{\frac{n-2}{4}+}\left( (2\bar{C})^{J} \| u_{0} \|_{\tilde{H}^{k}} \right) } \cdot
\end{array}
\nonumber
\end{equation}
Notice that such a partition always exists since, for $J$ large enough,

\begin{equation}
\begin{array}{ll}
\sum \limits_{j=1}^{J-1} \frac{\epsilon^{\frac{2(n+2)}{n-2}}}{g^{\frac{n+2}{2}+} \left( (2\bar{C})^{j} \| u_{0} \|_{\tilde{H}^{k}} \right) } & \gtrsim
\sum \limits_{j=1}^{J-1} \frac{1}{_{\log{( (2\bar{C})^{j} \| u_{0} \|_{\tilde{H}^{k}}) }} } \\
& = \sum \limits_{j=1}^{J-1} \frac{1}{_{j \log{(2\bar{C})} + \log {(\| u_{0} \|_{\tilde{H}^{k}})} }} \\
& \geq \| u \|^{\frac{2(n+2)}{n-2}}_{L_{t}^{\frac{2(n+2)}{n-2}} L_{x}^{\frac{2(n+2)}{n-2}} (I_{max}) }
\end{array}
\nonumber
\end{equation}
A continuity argument applied to  (\ref{Eqn:QJu}) shows that
$Y (I_1,u) \leq 2 \bar{C} \| u_0 \|_{\tilde{H}^{k}}$. By iteration $Y(I_j,u) \leq (2 \bar{C})^{j} \| u_0 \|_{\tilde{H}^{k}}$. Therefore there exists $Y_{max}$ such that $Y(I_{max},u) \leq Y_{max}$, proceeding similarly on $I_{max} \cap (-\infty,0]$. \\
We write $I_{max}=(a_{max},b_{max})$. Choose $\bar{t} < b_{max}$ close enough to $b_{max}$ so that
$ \| u \|_{L_{t}^{\frac{2(n+2)}{n-2}} L_{x}^{\frac{2(n+2)}{n-2}}([\bar{t},b_{max}))} \ll \delta$, with
$\delta$ defined in Proposition \ref{Prop:LocalWell}. Then there exists a large constant $C$ such that

\begin{equation}
\begin{array}{l}
\| e^{i (t -\bar{t}) \triangle} u(\bar{t}) \|_{L_{t}^{\frac{2(n+2)}{n-2}} L_{x}^{\frac{2(n+2)}{n-2}}([\bar{t},b_{max}))} \\
\leq \| u \|_{L_{t}^{\frac{2(n+2)}{n-2}} L_{x}^{\frac{2(n+2)}{n-2}}([\bar{t},b_{max}))}
+  C \left\| D \left( |u|^{\frac{4}{n-2}} u  g(|u|) \right) \right\|_{L_{t}^{\frac{2(n+2)}{n+4}} L_{x}^{\frac{2(n+2)}{n+4}}([\bar{t},b_{max}])} \\
\leq o(\delta) +  \left( o(\delta) \right)^{\frac{4}{n-2}-}  g(Y_{max})  \| D  u \|_{L_{t}^{\frac{2(n+2)}{n}} L_{x}^{\frac{2(n+2)}{n}}([\bar{t},b_{max}))}
\\
\leq \frac{3 \delta}{4} \cdot
\end{array}
\nonumber
\end{equation}
Also observe that $\| e^{i (t -\bar{t}) \triangle} u(\bar{t}) \|_{L_{t}^{\frac{2(n+2)}{n-2}} L_{x}^{\frac{2(n+2)}{n-2}}([\bar{t},\infty))}
\lesssim \| u(\bar{t}) \|_{\dot{H}^{1}} < \infty$. Hence by the monotone convergence theorem, there exists $\epsilon > 0$ such that \\
$ \| e^{i (t - \bar{t}) \triangle} u(\bar{t}) \|_{L_{t}^{\frac{2(n+2)}{n-2}} L_{x}^{\frac{2(n+2)}{n-2}} ([ \bar{t}, b_{max} + \epsilon])} \leq \delta $. Hence contradiction with Proposition \ref{Prop:LocalWell}.

\section{Proof of Theorem \ref{thm:main}}
\label{Sec:Thmmain}

Let $\theta \in (0,1)$ be a constant that is allowed to change from one line to the other one and such that all the estimates below are true. \\
\\
The proof is made of three steps:

\begin{itemize}

\item finite bound of $\| u \|_{L_{t}^{\frac{2(n+2)}{n-2}} L_{x}^{\frac{2(n+2)}{n-2}} (\mathbb{R})}$ for
$ 1 < k \leq \frac{n}{2}$ if $n \in \{3,4 \}$ and $1 <  k < \frac{4}{3}$ if $n=5$. By the monotone convergence theorem, (\ref{Eqn:SobolevIneq1}), the interpolation between $L_{t}^{\infty} L_{x}^{1_{2}^{*}}$ and $L_{t}^{\bar{Q}} L_{x}^{\bar{R}}$, and Proposition \ref{prop:BoundLong} it is enough to
find for all time $T \geq 0$ a finite bound of $ X([-T,T],u)$. In fact we shall prove that this bound does not depend on time $T$.
By time reversal symmetry $($ i.e if $ t \rightarrow u(t,x)$ is a solution of (\ref{Eqn:BarelySchrod}) then $t \rightarrow \bar{u}(-t,x)$ is also a solution of (\ref{Eqn:BarelySchrod}) $)$ we may WLOG restrict ourselves to $[0,T]$. We define

\begin{equation}
\begin{array}{ll}
\mathcal{F} & : = \left\{ T \in [0, \infty): \sup_{t \in [0,T]} X([0,t],u)  \leq M_{0}  \right \}
\end{array}
\end{equation}
We claim that $\mathcal{F}= [0,\infty)$ for $M_{0}$, a large constant (to be chosen later)  depending only on $ \| u_{0} \|_{\tilde{H}^{k}}$.
Indeed

\begin{itemize}

\item $0 \in \mathcal{F}$.

\item $\mathcal{F}$ is closed by continuity

\item $\mathcal{F}$ is open. Indeed let $T \in \mathcal{F}$. Then, by continuity there exists $\delta > 0$ such that for $T^{'} \in [0, T + \delta]$
we have  $ X ( [0,T^{'}])  \leq 2 M_{0} $. In view of (\ref{Eqn:BoundLong}), this implies, in particular, that

\begin{equation}
\begin{array}{ll}
\| u \|^{\bar{Q}}_{L_{t}^{\bar{Q}} L_{x}^{\bar{R}} ([0,T^{'}])} & \leq C_{1}^{C_1 g^{b_{n}+} ( 2 M_{0} )} \cdot
\end{array}
\label{Eqn:ControlApCrit}
\end{equation}
Let $J:= [0,a] \subset I_{max}$ be an interval such that  $\| u \|_{L_{t}^{\bar{Q}} L_{x}^{\bar{R}} (J)} \lesssim 1$. By (\ref{Eqn:Strich}) and Corollary \ref{Cor:Nonlin} we get $($ observe that $ \| D^{j} u \|_{L_{t}^{\infty-} L_{x}^{2+} (J)}  \lesssim X(J,u) $ for $ j \in \{ 1,k \} $: this follows by interpolation between
$ \| D^{j} u \|_{L_{t}^{2} L_{x}^{1_{2}^{*}} (J)} $ and $\| D^{j} u \|_{L_{t}^{\infty} L_{x}^{2}(J)}$ $)$

\begin{equation}
\begin{array}{ll}
X(J,u) & \lesssim \| u_0 \|_{\tilde{H}^{k}} +
 \left\| D \left( |u|^{\frac{4}{n-2}} u g(|u|) \right) \right\|_{L_{t}^{\breve{Q}} L_{x}^{\breve{R}}(J)}
+  \left\| D^{k} \left( |u|^{\frac{4}{n-2}} u g(|u|) \right) \right\|_{L_{t}^{\breve{Q}} L_{x}^{\breve{R}}(J)} \\
& \leq \bar{C} \| u_0 \|_{\tilde{H}^{k}} + 2 \bar{C} X(J,u) \| u \|^{\frac{4}{n-2}-}_{L_{t}^{\bar{Q}} L_{x}^{\bar{R}}(J)}
g \left( X(J,u) \right),
\end{array}
\nonumber
\end{equation}
where $\bar{C}$ is a fixed, large, and positive constant. \\
Let $ 0 < \epsilon \ll 1$. From the estimate above we see that if $J$ satisfies
$ \| u \|_{L_{t}^{\bar{Q}} L_{x}^{\bar{R}} (J)} = \frac{\epsilon}{g^{\frac{n-2}{4}+} ( 2 \bar{C} \| u_{0}
\|_{\tilde{H}^{k}} ) } $ then a simple continuity argument  shows that

\begin{equation}
\begin{array}{ll}
X(J,u)   & \leq 2 \bar{C} \| u_{0} \|_{\tilde{H}^{k}} \cdot
\end{array}
\nonumber
\end{equation}
We divide $[0,T^{'}]$ into subintervals $(J_{i})_{1 \leq i \leq I}$ such that $ \| u \|_{L_{t}^{\bar{Q}} L_{x}^{\bar{R}}
(J_{i})} = \frac{\epsilon}{g^{\frac{n-2}{4}+} \left( (2\bar{C})^{i} \| u_{0} \|_{\tilde{H}^{k}} \right) } $, $ 1 \leq i < I$ and  $ \| u
\|_{L_{t}^{\bar{Q}} L_{x}^{\bar{R}} (J_{I})} \leq \frac{\epsilon}{g^{\frac{n-2}{4}+} \left( (2\bar{C})^{I} \| u_{0}
\|_{\tilde{H}^{k}} \right) } $. Notice that such a partition exists by (\ref{Eqn:ControlApCrit}) and the following inequality

\begin{equation}
\begin{array}{ll}
C_{1}^{C_{1} g^{b_{n}+} (2M_{0})} & \gtrsim
\sum \limits_{i=1}^{I-1} \frac{1}{_{g^{ \frac{\bar{Q}(n-2)}{4}+} \left( (2\bar{C})^{i} \| u_{0} \|_{\tilde{H}^{k}} \right)}} \\
& \geq \sum \limits_{i=1}^{I-1} \frac{1} {_{
\log^{\frac{\bar{Q}(n-2)\gamma}{4} +} {\left( \log{( 10 + (2\bar{C})^{2i} \| u_{0} \|^{2}_{\tilde{H}^{k}} )} \right)}}} \\
& \gtrsim  \sum \limits_{i=1}^{I-1} \frac{1} {_{ \log^{\frac{\bar{Q}(n-2)\gamma}{4}+} \left( 2i \log{(2\bar{C})} + 2 \log{ ( \| u_{0} \|_{\tilde{H}^{k}} ) } \right)  }} \\
& \gtrsim  \sum \limits_{i=1}^{I-1} \frac{1}{ i^{\frac{1}{2}}} \\
& \gtrsim  I^{\frac{1}{2}} \cdot \\
\end{array}
\label{Eqn:EstI}
\end{equation}
Moreover, by iterating over $i$  we get

\begin{equation}
\begin{array}{ll}
X([0,T^{'}],u) & \leq (2\bar{C})^{I+1} \| u_{0} \|_{\tilde{H}^{k}}
\end{array}
\nonumber
\end{equation}
Therefore by (\ref{Eqn:EstI}) there exists a positive constant $C^{'}$

\begin{equation}
\begin{array}{l}
\log{I}  \lesssim \log{ (C^{'})}  + C_{1} \log^{(b_{n}+) \gamma}  \left( \log{(10 + 4 M^{2}_{0})}   \right) \log{( C_{1} )}
\end{array}
\nonumber
\end{equation}
and for $ M_{0}$ large enough

\begin{equation}
\begin{array}{l}
\log{ (C^{'})}  + C_{1} \log^{(b_{n} +) \gamma}  \left( \log{(10 + 4 M^{2}_{0})}   \right) \log{ \left(  C_{1} \right)}
\ll \log { \left( \frac{  \log{ \left( \frac{M_{0}}{ \| u_{0} \|_{\tilde{H}^{k}}} \right) }} {_{\log{(2 \bar{C})}}} \right)}
\end{array}
\nonumber
\end{equation}
since (recall that $\gamma < \frac{1}{b_{n}}$)

\begin{equation}
\begin{array}{ll}
\frac{ \log{ (C^{'})}  + C_{2} \log^{(b_{n}+) \gamma}  \left( \log{(10 + 4 M^{2}_{0})}   \right) \log{(C_{1}) }} {\log { \left( \frac{  \log{ \left( \frac{M_{0}}{ \| u_{0} \|_{\tilde{H}^{k}}} \right) }}
{_{\log{(2 \bar{C})}}} \right)}} &
\rightarrow_{M_{0} \rightarrow \infty} 0 \cdot
\end{array}
\nonumber
\end{equation}

\end{itemize}
Hence $X(\mathbb{R},u) < \infty$. Observe that this implies a finite bound of
$\| D^{j} u ||_{L_{t}^{\frac{2(n+2)}{n}} L_{x}^{\frac{2(n+2)}{n}} (\mathbb{R})}$, $j \in \{1,k \}$, since there exists
$\theta \in[0,1]$ such that

\begin{equation}
\begin{array}{ll}
\| D^{j} u \|_{L_{t}^{\frac{2(n+2)}{n}} L_{x}^{\frac{2(n+2)}{n}}(\mathbb{R})} & \lesssim
\| D^{j} u \|^{\theta}_{L_{t}^{2} L_{x}^{1_{2^{*}}} (\mathbb{R})}  \| D^{j} u \|^{1- \theta}_{L_{t}^{\infty} L_{x}^{2}(\mathbb{R})} \cdot
\end{array}
\label{Eqn:InterpoDju}
\end{equation}

\item Finite bound of $Y(\mathbb{R},u)$ for all $k \in I_{n}$: this follows from Proposition \ref{Prop:PersReg}.

\item Scattering: it is enough to prove that $e^{- i t \triangle} u(t)$ has a limit as $t \rightarrow \infty$ in $\tilde{H}^{k}$. If $t_1$ is large enough and $t_{1}< t_{2}$
then by Corollary \ref{Cor:Nonlin}, Proposition \ref{Prop:EstHighReg}, and by dualizing (\ref{Eqn:Strich}) with $G=0$ (more precisely the estimate
$\| D^{j} u \|_{L_{t}^{\frac{2(n+2)}{n}} L_{x}^{\frac{2(n+2)}{n}} ([t_1,t_2])} \lesssim \| u_0 \|_{\dot{H}^{j}}$ if $j \in \{1,k \}$ ) we get

\begin{equation}
\begin{array}{l}
\| e^{-i t_{1} \triangle} u(t_{1}) - e^{- i t_{2} \triangle} u(t_{2}) \|_{\tilde{H}^{k}} \\
 \lesssim \| D^{k} \left( |u|^{\frac{4}{n-2}} u g(|u|)  \right) \|_{L_{t}^{\frac{2(n+2)}{n+4}} L_{x}^{\frac{2(n+2)}{n+4}} ([t_{1},t_{2}]) } +
\| D \left( |u|^{\frac{4}{n-2}} u g(|u|)  \right) \|_{L_{t}^{\frac{2(n+2)}{n+4}} L_{x}^{\frac{2(n+2)}{n+4}} ([t_{1},t_{2}]) } \\
 \lesssim \| u \|^{\frac{4}{n-2}-}_{L_{t}^{\frac{2(n+2)}{n-2}} L_{x}^{\frac{2(n+2)}{n-2}} ([t_{1},t_{2}]) }
\end{array}
\nonumber
\end{equation}
and  we conclude that given $\epsilon > 0$ there exists $A(\epsilon)$ large enough such that if $t_{2} \geq t_{1} \geq A(\epsilon)$ then $ \| e^{-i t_{1} \triangle} u(t_{1}) - e^{- i t_{2} \triangle} u(t_{2}) \|_{\tilde{H}^{k}} \leq \epsilon  $. The Cauchy criterion is
satisfied. Hence scattering.

\end{itemize}

\section{Proof of Proposition \ref{prop:BoundLong}} \label{Sec:Boundlong}

In this section we prove Proposition \ref{prop:BoundLong}. \\
\\
Let $(\alpha,\beta,\delta)$ be defined as follows:

\begin{equation}
(\alpha,\beta,\delta):=
\left\{
\begin{array}{l}
\left( \frac{n-2}{2}-, - \left(  \frac{n-2}{2} + \right) , \frac{n-2}{n}+ \right) \; \text{if} \; n \in \{3,4\} \\
\left( 1- ,-(1+), \frac{3}{5}+ \right) \; \text{if} \; n=5 \cdot
\end{array}
\right.
\nonumber
\end{equation}

The proof relies upon a Morawetz-type estimate:

\begin{lem}
Let $u$ be an $\tilde{H}^{k}$ $-$ solution of (\ref{Eqn:BarelySchrod}) on a compact interval $I$. Let $A > 1$. Then

\begin{equation}
\begin{array}{ll}
\int_{I} \int_{|x| \leq A |I|^{\frac{1}{2}}} \frac{ \tilde{F}(u,\bar{u}) (t,x)}{|x|} dx \, dt & \lesssim E  A  |I|^{\frac{1}{2}}
\end{array}
\label{Eqn:MorawEst}
\end{equation}
with

\begin{equation}
\begin{array}{ll}
\tilde{F}(u,\bar{u})(t,x) & := \int_{0}^{|u|(t,x)} s^{\frac{n+2}{n-2}} \left( \frac{4}{n-2} g(s) + s g^{'}(s) \right) \, ds
\end{array}
\label{Eqn:DeftildeF}
\end{equation}
\label{lem:Morawest}
\end{lem}

\begin{rem}
If $k > \frac{n}{2}$ then the proof of (\ref{Eqn:MorawEst}) is in \cite{triroyschrod}. If not, it is mostly
contained in \cite{triroyschrod}. Indeed, the proof relies on integration by parts of the local momentum
identity multiplied by an appropriate weight. In the case where $n \in \{3,4 \}$, the integration by parts
holds for smooth solutions of (\ref{Eqn:BarelySchrod}) (i.e solutions in $\tilde{H}^{p}$ with exponents $p$
large enough). Then (\ref{Eqn:MorawEst}) holds
for $\tilde{H}^{k}$ solutions for  $k \in I_n$ by a standard approximation argument with smooth solutions.
If $n=5$ then the nonlinearity  is not that smooth: its derivatives of $(z, \bar{z}) \rightarrow |z|^{\frac{4}{n-2}} z $  are not even twice differentiable. So
one should first smooth out the nonlinearity, obtain an identity similar to the local momentum identity
for smooth solutions (i.e solutions lying in Sobolev spaces with large exponent) of the ``smoothed'' equation and
then take limit in $\tilde{H}^{k}$ for $k \in I_n$ by again a standard approximation argument with smooth solutions.
\end{rem}

We prove now Proposition \ref{prop:BoundLong}. The proof follows closely an argument in \cite{taorad} (this argument was also used
in \cite{triroyschrod}) and it is based upon methods of concentration (see e.g \cite{bourg}). \\

\fbox{\textbf{Step 1}} \\

We divide the interval $J=[t_{1},t_{2}]$ into subintervals $(J_{l}:=[\bar{t}_{l},\bar{t}_{l+1}])_{1  \leq  l \leq L}$ such that

\begin{equation}
\begin{array}{ll}
\| u \|^{\bar{Q}}_{L_{t}^{\bar{Q}}  L_{x}^{\bar{R}} (J_{l}) } & = \eta_{1}
\end{array}
\end{equation}
and

\begin{equation}
\begin{array}{ll}
\| u \|^{\bar{Q}}_{L_{t}^{\bar{Q}}  L_{x}^{\bar{R}} (J_{L}) } & \leq \eta_{1},
\end{array}
\label{Eqn:LastConc}
\end{equation}
with $ 0 < c_{1} \ll 1$ and $\eta_{1} := \frac{c_{1}}{_{g^{\frac{(n-2) \bar{Q}}{6-n}+} (M) }} $. It is enough to find an upper bound of $L$. In view of (\ref{Eqn:BoundLong}) , we may replace WLOG the $``\leq''$ sign with the $``=''$ sign in
(\ref{Eqn:LastConc}).  \\
Notice that the value of this parameter, along with the values of the other parameters $\eta_{2}$, $\eta_{3}$ and $\eta$ are chosen so that
all the constraints appearing in the process to find an upper bound of $L$ are satisfied and so that $L \eta_{1}$ is as small as possible.\\
\\
\fbox {\textbf{Step 2}} \\

We first prove that some norms on these intervals $J_{l}$ are bounded.

\begin{res}
We have
\begin{equation}
\begin{array}{ll}
\| D u \|_{L_{t}^{\infty-} L_{x}^{2+} (J_{l})} & \lesssim 1
\end{array}
\end{equation}
\label{res:ControlDu}
\end{res}

\begin{proof}

From (\ref{Eqn:Strich}), the conservation of the energy, Corollary \ref{Cor:Nonlin}, and
Proposition \ref{Prop:NonlinFracSmooth} combined with (\ref{Eqn:Sobklarge}), we get

\begin{equation}
\begin{array}{ll}
\| D u \|_{L_{t}^{\infty-} L_{x}^{2+} (J_{l})} & \lesssim \| D u (\bar{t}_{l}) \|_{L^{2}} +
\left\| D ( |u|^{\frac{4}{n-2}} u g(|u|) ) \right\|_{L_{t}^{\breve{Q}} L_{x}^{\breve{R}} (J_{l}) } \\
& \lesssim 1 + \| D u \|_{L_{t}^{\infty-} L_{x}^{2+}(J_l)}
\| u \|^{\frac{4}{n-2}-}_{L_{t}^{\bar{Q}} L_{x}^{\bar{R}} (J_l)} g(M)
\end{array}
\nonumber
\end{equation}
Therefore, by a continuity argument $($ Observe that the above estimate also holds if $J_l$ is replaced with $K_l$ where $K_l \subset J_l$ $)$, we conclude that $\| D u \|_{L_{t}^{\infty-} L_{x}^{2+} (J_{l})} \lesssim  1$.

\end{proof}

\begin{res}
Let $\tilde{J}:= [ \tilde{t}_1, \tilde{t}_2 ] \subset J_l$ be such that

\begin{equation}
\begin{array}{l}
\frac{\eta_{1}}{2} \leq \| u \|^{\bar{Q}}_{L_{t}^{\bar{Q}} L_{x}^{\bar{R}} (\tilde{J})} \leq \eta_{1}
\end{array}
\end{equation}
Then
\begin{equation}
\begin{array}{ll}
\| u_{l,\tilde{t}_{j}} \|^{\bar{Q}}_{ L_{t}^{\bar{Q}}  L_{x}^{\bar{R}} (\tilde{J})} & \gtrsim \eta_{1}
\end{array}
\label{Eqn:LowerBoundLinRes}
\end{equation}
for $j \in \{1,2 \}$. \label{res:linearpartsubs}
\end{res}

\begin{proof}
From Result \ref{res:ControlDu} we get

\begin{equation}
\begin{array}{ll}
\left\| u - u_{l,\tilde{t}_{j}} \right\|_{L_{t}^{\bar{Q}} L_{x}^{\bar{R}} (\tilde{J})}  & \lesssim \left\| D( |u|^{\frac{4}{n-2}} u g(|u|) )
\right\|_{L_{t}^{\breve{Q}} L_{x}^{\breve{R}} (\tilde{J}) } \\
& \lesssim \| D u \|_{L_{t}^{\infty-}  L_{x}^{2^{+}} (\tilde{J}) } \| u \|_{L_{t}^{\bar{Q}}
L_{x}^{\bar{R}} (\tilde{J})}^{\frac{4}{n-2}-} g(M) \\
& \lesssim \|  u \|_{L_{t}^{\bar{Q}} L_{x}^{\bar{R}} (\tilde{J})}^{\frac{4}{n-2}-} g(M) \\
& \ll \eta_{1}^{\frac{1}{_{\bar{Q}}}} \cdot
\end{array}
\end{equation}
Therefore (\ref{Eqn:LowerBoundLinRes}) holds.

\end{proof}

\fbox{\textbf{Step 3}} \\

Let

\begin{equation}
\begin{array}{l}
\eta_{2}:= c_{2}
\eta_{1}^{
\left(  1  + n \left( \frac{1}{_{2}} - \frac{1}{_{\bar{R}}} \right) \right)
\left( \frac{1}{1 - \frac{_{1_{2}^{*}}}{_{\bar{R}}}}
+ \frac{n}{2 \alpha} \right) } g^{- \frac{n \delta \bar{Q}}{2 \alpha} \left( 1 + n \left( \frac{1}{_{2}} - \frac{1}{_{\bar{R}}} \right) \right) } (M)
\end{array}
\label{Eqn:Dfneta2}
\end{equation}
with $0 < c_{2} \ll c_{1}$. An interval $J_{l_{0}} = [\bar{t}_{l_{0}}, \bar{t}_{l_{0}+1}]$ of the partition $(J_{l})_{ 1 \leq l \leq L}$ is exceptional if

\begin{equation}
\begin{array}{ll}
\| u_{l,t_{1}} \|^{\bar{Q}}_{L_{t}^{\bar{Q}} L_{x}^{\bar{R}} (J_{l_{0}})}
+ \| u_{l,t_{2}} \|^{\bar{Q}}_{L_{t}^{\bar{Q}} L_{x}^{\bar{R}} (J_{l_{0}}) } & \geq \eta_{2}
\end{array}
\end{equation}
Notice that, in view of (\ref{Eqn:Strich}) and (\ref{Eqn:SobolevIneq1}), it is easy to find an upper bound of the cardinal of the exceptional
intervals:

\begin{equation}
\begin{array}{ll}
\card{\{ J_{l}: \, J_{l} \, \mathrm{exceptional} \}} & \lesssim \eta_{2}^{-1}
\end{array}
\label{Eqn:BoundCardExcep}
\end{equation}

\fbox{ \textbf{Step 4} } \\

Now we prove that on  each unexceptional subintervals $J_{l}$ there is a ball for which we have a mass concentration.

\begin{res}
There exists an $x_{l} \in \mathbb{R}^{n}$, two positive constants $c_3 \ll 1 $ and $C_3 \gg 1$ such that for each unexceptional
interval $J_{l}$ and for $t \in J_{l}$

\begin{equation}
\begin{array}{ll}
\textrm{Mass} \left( u(t), B(x_{l}, C_3  g^{\gamma_3} (M) |J_{l}|^{\frac{1}{2}} )   \right) & \geq c_3 g^{-
\gamma_3} (M) |J_{l}|^{\frac{1}{2}}
\end{array}
\label{Eqn:ConcentrationMassn3}
\end{equation}
with

\begin{equation}
\begin{array}{l}
\gamma_{3} := \left( \frac{n-2}{6-n} \left( \frac{1}{1 - \frac{1_{2}^{*}}{_{\bar{R}}}} + \frac{n}{2 \alpha} \right) + \frac{n \delta}{2 \alpha} \right) + \cdot
\end{array}
\nonumber
\end{equation}

\label{res:ConcentrationMass}
\end{res}

\begin{proof}

By time translation invariance $($ i.e if $u$ is a solution of (\ref{Eqn:BarelySchrod}) and $t_{0} \in \mathbb{R}$ then $(t,x) \rightarrow
u(t-t_{0},x)$ is also a solution of (\ref{Eqn:BarelySchrod}) $)$ we may assume that $\bar{t}_{l}=0$. By using the pigeonhole principle and the
reflection symmetry (if necessary) $($  if $u$ is a solution of (\ref{Eqn:BarelySchrod})
 then $ (t,x) \rightarrow \bar{u}(-t,x) $ is also a solution of (\ref{Eqn:BarelySchrod}) $)$ we may assume that

\begin{equation}
\begin{array}{ll}
\| u \|^{\bar{Q}}_{L_{t}^{\bar{Q}} L_{x}^{\bar{R}} \left(\frac{|J_l|}{2}, |J_l| \right)} & \geq \frac{\eta_{1}}{4}
\end{array}
\label{Eqn:CriticalNormLarge}
\end{equation}
By the pigeonhole principle there exists $t_{*} $ such that $[(t_{*} -\eta_{3})|J_{l}|, t_{*} |J_{l}| ] \subset \left[0, \frac{|J_{l}|}{2}
\right] $ (with $ 0 < \eta_{3} \ll 1$) with

\begin{equation}
\begin{array}{ll}
\int_{(t^{*} - \eta_{3}) |J_{l}|}^{t_{*} | J_{l}|} \left( \int_{\mathbb{R}^{n}}  | u(t,x)|^{\bar{R}} \, dx \right)^{\frac{\bar{Q}}{_{\bar{R}}}} \, dt \lesssim \eta_{1} \eta_{3}, \; \text{and}
\end{array}
\end{equation}

\begin{equation}
\begin{array}{ll}
\left( \int_{\mathbb{R}^{n}} | u_{l,t_{1}} ( (t_{*} -\eta_{3}) |J_{l}|,x ) |^{\bar{R}} \, dx \right)^{\frac{\bar{Q}}{_{\bar{R}}}} & \lesssim \frac{\eta_{2}}{|J_{l}|}\cdot
\end{array}
\label{Eqn:Ptwiseult1}
\end{equation}
Applying Result \ref{res:linearpartsubs} to (\ref{Eqn:CriticalNormLarge}) we have

\begin{equation}
\begin{array}{ll}
\int_{t_{*} |J_{l}|}^{|J_{l}|}
\left(
\int_{\mathbb{R}^{n}} | e^{i(t -t_{*}|J_{l}|) \triangle} u(t^{*}|J_{l}|,x) |^{\bar{R}} \, dx
\right)^{\frac{\bar{Q}}{_{\bar{R}}}}
 \, dt
& \gtrsim \eta_{1} \cdot
\end{array}
\label{Eqn:LowerBoundLin}
\end{equation}
By Duhamel formula we have

\begin{equation}
\begin{array}{ll}
u(t_{*} |J_{l}|) & = e^{i(t_{*}|J_{l}|- t_{1} ) \triangle} u(t_{1}) - i \int_{t_{1}}^{ (t_{*} - \eta_{3}) |J_{l}|  } e^{i(t_{*} |J_{l}|-s)
\triangle} (
|u(s)|^{\frac{4}{n-2}} u(s) g(|u(s)|) ) \, ds \\
& - i \int_{(t_{*} - \eta_{3}) |J_{l}|}^{t_{*}|J_{l}|} e^{i(t_{*} |J_{l}|-s) \triangle} ( |u(s)|^{\frac{4}{n-2}} u(s) g(|u(s)|) ) \, ds
\end{array}
\end{equation}
and, composing this equality with $e^{i(t-t_{*}|J_{l}|) \triangle}$ we get

\begin{equation}
\begin{array}{ll}
e^{i(t-t_{*}|J_{l}|) \triangle} u(t_{*} |J_{l}|) & = u_{l,t_{1}}(t) - i \int_{t_{1}}^{ (t_{*} - \eta_{3}) |J_{l}|  } e^{i(t-s)
\triangle} ( |u(s)|^{\frac{4}{n-2}} u(s) g(|u(s)|) ) \, ds \\
& - i \int_{(t_{*} - \eta_{3}) |J_{l}|}^{t_{*}|J_{l}|} e^{i(t -s) \triangle} ( |u(s)|^{\frac{4}{n-2}} u(s) g(|u(s)|) ) \, ds \\
& = u_{l,t_{1}}(t) + v_{1}(t) + v_{2}(t)
\end{array}
\label{Eqn:Decompulti}
\end{equation}
We get from a variant of the Strichartz estimates (\ref{Eqn:Strich})

\begin{equation}
\begin{array}{l}
\| v_{2} \|_{L_{t}^{\bar{Q}} L_{x}^{\bar{R}}  ([t_{*} |J_{l}|, |J_{l}|])
\cap L_{t}^{\infty} D^{-1} L_{x}^{2}([t_{*} |J_{l}|, |J_{l}|])} \\
\lesssim \left\| D \left( |u|^{\frac{4}{n-2}} u g(|u|) \right) \right\|_{L_{t}^{\breve{Q}}
L_{x}^{\breve{R}}([ (t_{*} - \eta_{3}) |J_{l}|, t^{*} |J_{l}| ])} \\
\lesssim \|  D u  \|_{L_{t}^{\infty-} L_{x}^{2+} ([ (t_{*} -\eta_{3}) |J_{l}|,  t^{*}|J_{l}| ])  } \| u
\|^{\frac{4}{n-2}-}_{L_{t}^{\bar{Q}} L_{x}^{\bar{R}} ([ (t_{*} - \eta_{3}) |J_{l}|, t^{*} |J_{l}| ]) } g(M) \\
\lesssim (\eta_{1} \eta_{3})^{\frac{4}{_{\bar{Q}(n-2)}}-}  g(M) \\
 \ll \eta_{1}^{\frac{1}{_{\bar{Q}}}}
\end{array}
\label{Eqn:Dv2Ineq}
\end{equation}
Notice also that $\eta_{2} \ll \eta_{1}$ and that $J_{l}$ is non-exceptional. Therefore $ \| u_{l,t_{1}} \|^{\bar{Q}}_{L_{t}^{\bar{Q}}
L_{x}^{\bar{R}} ([t_{*} |J_{l}|, |J_{l}| ]) } \ll \eta_{1}$ and combining this inequality with (\ref{Eqn:Dv2Ineq}) and
(\ref{Eqn:LowerBoundLin}) we conclude that the  $ L_{t}^{\bar{Q}} L_{x}^{\bar{R}} $ norm of $v_{1}$ on $[t_{*} |J_{l}|,
|J_{l}| ]$ is bounded from below:

\begin{equation}
\begin{array}{ll}
\| v_{1} \|^{\bar{Q}}_{L_{t}^{\bar{Q}} L_{x}^{\bar{R}} ([t_{*} |J_{l}|, |J_{l}| ])} & \gtrsim \eta_{1}
\end{array}
\end{equation}
By (\ref{Eqn:Decompulti}) and (\ref{Eqn:Dv2Ineq}) we also have an upper bound of the  $ L_{t}^{\bar{Q}}
L_{x}^{\bar{R}} $ norm of $v_{1}$ on $[t_{*} |J_{l}|, |J_{l}| ]$

\begin{equation}
\begin{array}{ll}
\| v_{1} \|^{\bar{Q}}_{L_{t}^{\bar{Q}} L_{x}^{\bar{R}}  ([t_{*} |J_{l}|, |J_{l}|])
\cap L_{t}^{\infty} D^{-1} L_{x}^{2}([t_{*} |J_{l}|, |J_{l}|])  } &  \lesssim 1
\end{array}
\end{equation}
Now we use a lemma that is proved in Subsection \ref{Subsec:Lemmaregv}. This lemma provides some information regarding the regularity of $v_1$.

\begin{lem}
We have

\begin{equation}
\begin{array}{ll}
\| v_{1,h} - v_{1} \|_{L_{t}^{\infty}  L_{x}^{\bar{R}} ( [t_{*} |J_{l}|, |J_{l}|]) } & \lesssim |h|^{\alpha} |J_{l}|^{\beta}
g^{\delta}(M) \cdot
\end{array}
\label{Eqn:Regv}
\end{equation}

\label{lem:regv}
\end{lem}
Denote by $v_{1,h}^{av}(x) := \int \chi(y) v_{1}(x + |h|y) \; dy $ with  $\chi$ a bump function with total mass equal to one and such that
$\supp (\chi) \subset B(0,1)$. Then

\begin{equation}
\begin{array}{ll}
\| v_{1,h}^{av} - v_{1} \|_{L_{t}^{\bar{Q}} L_{x}^{\bar{R}} ([ t_{*}|J_{l}|,|J_{l}|])  } & \lesssim |J_{l}|^{\frac{1}{_{\bar{Q}}}}
\| v_{1,h}^{av} - v_{1} \|_{L_{t}^{\infty} L_{x}^{\bar{R}} ([ t_{*}|J_{l}|,|J_{l}|]) } \\
& \lesssim |h|^{\alpha} | J_{l}|^{\beta + \frac{1}{_{\bar{Q}}}} g^{\delta}(M) \cdot
\end{array}
\end{equation}
Therefore if $h$ satisfies $ |h| := c_{3} \eta_{1}^{\frac{1}{_{\bar{Q} \alpha}}} |J_{l}|^{- \left( \beta + \frac{1}{_{\bar{Q}}} \right) \frac{1}{\alpha}}
g^{- \frac{\delta}{\alpha}}(M)$
with $ 0 < c_{3} \ll 1 $ then

\begin{equation}
\begin{array}{ll}
\| v_{1,h}^{av} \|^{\bar{Q}}_{L_{t}^{\bar{Q}} L_{x}^{\bar{R}} ([ t_{*}|J_{l}|,|J_{l}|]) } & \gtrsim \eta_{1} \cdot
\end{array}
\label{Eqn:Ineqv1hBd}
\end{equation}
Now notice that by the Duhamel formula $v_{1}(t) = u_{l,(t_{*} - \eta_{3})|J_{l}|}(t) - u_{l,t_1} (t)  $ and therefore  $\| v_{1} \|_{L_{t}^{\infty} L_{x}^{1_{2}^{*}} ([t^{*} |J_{l}|, |J_{l}|])}
\lesssim  1 $. From that we get $ \| v_{1,h}^{av} \|_{L_{t}^{\frac{\bar{Q} 1_{2}^{*}}{\bar{R}}} L_{x}^{1_{2}^{*}} ([t^{*} |J_{l}|, |J_{l}|]) } \lesssim
|J_{l}|^{\frac{\bar{R}}{_{\bar{Q} 1_{2}^{*}}}}$ and, by interpolation,

\begin{equation}
\begin{array}{ll}
\| v_{1,h}^{av} \|_{L_{t}^{\bar{Q}} L_{x}^{\bar{R}}([t^{*} |J_{l}|, |J_{l}|]) } & \lesssim \| v_{1,h}^{av}
\|^{1 - \frac{1_{2}^{*}}{_{\bar{R}}}}_{L_{t}^{\infty} L_{x}^{\infty}([t^{*} |J_{l}|, |J_{l}|]) }
\| v_{1,h}^{av} \|^{\frac{1_{2}^{*}}{_{\bar{R}}}}_{L_{t}^{\frac{\bar{Q} 1_{2}^{*}}{_{\bar{R}}}} L_{x}^{1_{2}^{*}} ([t^{*} |J_{l}|, |J_{l}|]) } \cdot
\end{array}
\end{equation}
Hence, in view of (\ref{Eqn:Ineqv1hBd})

\begin{equation}
\begin{array}{ll}
\| v_{1,h}^{av} \|_{L_{t}^{\infty} L_{x}^{\infty} ([ t_{*}|J_{l}|,|J_{l}|])} & \gtrsim
\left( \frac{\eta_1^{\frac{1}{_{\bar{Q}}}}}{|J_l|^{\frac{1}{_{\bar{Q}}} }} \right)^{\frac{1}{1 - \frac{1_{2}^{*}}{_{\bar{R}}}}}
\end{array}
\label{Eqn:LowerBoundLtinfLxinf}
\end{equation}
Writing $ \textrm{Mass}(v_1(t),B(x,r))  = r^{\frac{n}{2}}  \left( \int_{|y| \leq 1} | v_1(t,x + r y)    |^{2} \, dy \right)^{\frac{1}{2}} $ we deduce from
Cauchy Schwartz and (\ref{Eqn:LowerBoundLtinfLxinf}) that there exists  $\check{t}_l \in [t_{*}|J_{l}|, |J_{l}| ]$ and $x_{l} \in \mathbb{R}^{n}$ such
that

\begin{equation}
\begin{array}{ll}
\textrm{Mass} \left( v_{1}(\check{t}_{l}),B(x_{l},|h|) \right) & \gtrsim
\left( \frac{\eta_1^{\frac{1}{_{\bar{Q}}}}}{|J_l|^{\frac{1}{_{\bar{Q}}} }} \right)^{\frac{1}{1 - \frac{1_{2}^{*}}{_{\bar{R}}}}}
|h|^{\frac{n}{2}} \cdot
\end{array}
\end{equation}
Therefore, by (\ref{Eqn:UpBdDerivM}) we see that if $R = C_{3} |J_{l}|
\left( \frac{\eta_1^{\frac{1}{_{\bar{Q}}}}}{|J_l|^{\frac{1}{_{\bar{Q}}} }} \right)^{- \frac{1}{1 - \frac{1_{2}^{*}}{_{\bar{R}}}}}
|h|^{-\frac{n}{2}} $ with $C_3 \gg 1$ then

\begin{equation}
\begin{array}{ll}
\textrm{Mass} \left( v_{1}( (t_{*} - \eta_{3}) |J_{l}| ), B(x_{l}, R)   \right) & \gtrsim
\left( \frac{\eta_1^{\frac{1}{_{\bar{Q}}}}}{|J_l|^{\frac{1}{_{\bar{Q}}} }} \right)^{\frac{1}{1 - \frac{1_{2}^{*}}{_{\bar{R}}}}}
|h|^{\frac{n}{2}} \cdot
\end{array}
\end{equation}
Notice that $ u \left( (t_{*} - \eta_{3})|J_{l}|  \right)= u_{l,t_{1}} \left( (t_{*} - \eta_{3})|J_{l}|  \right) - i v_{1} \left( (t_{*} -
\eta_{3}) |J_{l}|  \right) $. By H\"older inequality, (\ref{Eqn:Dfneta2}), and (\ref{Eqn:Ptwiseult1})

\begin{equation}
\begin{array}{ll}
\textrm{Mass} \left( u_{l,t_{1}} ( (t_{*} - \eta_{3}) |J_{l}|), B(x_{l},R) \right) & \lesssim
R^{n \left( \frac{1}{_{2}} - \frac{1}{_{\bar{R}}} \right)} \left( \frac{\eta_{2}}{|J_l|} \right)^{\frac{1}{_{\bar{Q}}}} \cdot \\
& \\
& \ll \left( \frac{\eta_1^{\frac{1}{_{\bar{Q}}}}}{|J_l|^{\frac{1}{_{\bar{Q}}} }} \right)^{\frac{1}{1 - \frac{1_{2}^{*}}{_{\bar{R}}}}}
|h|^{\frac{n}{2}} \cdot
\end{array}
\end{equation}
Therefore $ \textrm{Mass} \left( u ((t_{*} - \eta_{3}) |J_{l}|) ,B(x_{l},R) \right) \approx  \textrm{Mass} \left( v_{1}( (t_{*} - \eta_{3}) |J_{l}| ), B(x_{l}, R)
\right) $. Applying again (\ref{Eqn:UpBdDerivM}) we get

\begin{equation}
\begin{array}{ll}
\textrm{Mass} \left( u (t), B(x_{l},R) \right) & \gtrsim \left( \frac{\eta_1^{\frac{1}{_{\bar{Q}}}}}{|J_l|^{\frac{1}{_{\bar{Q}}} }} \right)^{\frac{1}{1 - \frac{1_{2}^{*}}{_{\bar{R}}}}}
|h|^{\frac{n}{2}}
\end{array}
\end{equation}
for $t \in J_{l}$. Putting everything together we get (\ref{Eqn:ConcentrationMassn3}).

\end{proof}

Next we use the radial symmetry to prove that, in fact, there is a mass concentration around the origin. \\

\fbox{ \textbf{Step 5} } \\

\begin{res}
There exists a positive constant $c_4 \ll 1$ and a
constant $C_4 \gg 1$  such that on each unexceptional interval $J_{l}$ we have

\begin{equation}
\begin{array}{ll}
\textrm{Mass} \left( u(t),  B(0, C_4 g^{\gamma_4^{a}}(M) |J_{l}|^{\frac{1}{2}}  )  \right) & \geq c_4
g^{-\gamma_4^{b}}(M) |J_{l}|^{\frac{1}{2}},
\end{array}
\label{Eqn:ConcMassZero}
\end{equation}
with $ \left( \gamma_4^{a}, \gamma_4^{b} \right) := \left( \gamma_3 ( 2 \times 1_2^{*}+1), \gamma_{3} \right)$.

\label{res:ConcMassZero}
\end{res}

\begin{proof}

Let $A:= C_4 g^{\gamma_3 (2 \times 1_2^{*}+1)}(M)$ for $C_4 \gg C_3$ large enough so that all the statements below are true. There are (a priori) two options:

\begin{itemize}

\item $ |x_{l}| \geq \frac{A}{2} |J_{l}|^{\frac{1}{2}} $ . Then there are at least $ \frac{A}{ 100 C_3 g^{\gamma_3}
(M) } $ rotations of the ball $ B(x_{l}, C_3 g^{\gamma_3} (M) |J_{l}|^{\frac{1}{2}} ) $ that are disjoint. Now,
since the solution is radial, the mass on each of these balls $B_{j}$ is equal to that of the ball
$ B(x_{l}, C_3  g^{\gamma_3} (M) |J_{l}|^{\frac{1}{2}}) $. But then by H\"older inequality we have

\begin{equation}
\begin{array}{ll}
\| u(t)\|^{1_{2}^{*}}_{L^{2}(B_{j})} & \leq \| u(t) \|^{1_{2}^{*}}_{L^{1_{2}^{*}}(B_{j})} \left( C_3 g^{\gamma_3} (M) |J_{l}|^{\frac{1}{2}} \right)^{1_{2}^{*}}
\end{array}
\end{equation}
and summing over $j$ we see from the estimate $\| u(t) \|^{1_{2}^{*}}_{L^{1_{2}^{*}}} \lesssim 1 $ that

\begin{equation}
\begin{array}{l}
\frac{A}{ 100 C_3 g^{\gamma_3} (M) } \left( c_3 g^{-\gamma_3}(M)
|J_{l}|^{\frac{1}{2}} \right)^{1_{2}{*}}  \\
\lesssim   \left( C_3 g^{\gamma_3} (M) |J_{l}|^{\frac{1}{2}} \right)^{1_{2}^{*}}
\end{array}
\end{equation}
must be true. But with the value of $A$ chosen above we see that this inequality cannot be satisfied. Therefore
this scenario is impossible.

\item $|x_{l}| \leq \frac{A}{2} |J_{l}|^{\frac{1}{2}}  $. Then by
(\ref{Eqn:ConcentrationMassn3}) and the triangle inequality, we see that (\ref{Eqn:ConcMassZero}) holds.
\end{itemize}

\end{proof}

\fbox{\textbf{Step 6}} \\

Combining the inequality (\ref{Eqn:ConcMassZero}) with the Morawetz-type inequality in Lemma \ref{lem:Morawest} we can prove that at least
one of the intervals $J_{l}$ is large. More precisely

\begin{res}
There exists a positive constant $c_5 \ll 1$ and $\tilde{l} \in [ 1,..,L ]$ such that

\begin{equation}
\begin{array}{ll}
|J_{\tilde{l}}|  &  \geq c_5 g^{- \gamma_5}(M) |J|,
\end{array}
\label{Eqn:DistribInterv}
\end{equation}
with $\gamma_5 :=  2 \left( \gamma_4^{a} \frac{3n-2}{n-2} + \gamma_{4}^{b} 1_{2}^{*} \right) $.

\label{res:DistribInterv}
\end{res}

\begin{proof}

There are two options:

\begin{itemize}

\item $J_{l}$ is unexceptional. Let $R:=  C_4 g^{\gamma_{4}^{a}}(M)  |J_{l}|^{\frac{1}{2}} $. By H\"older inequality
(w.r.t space) and by integration in time we have

\begin{equation}
\begin{array}{ll}
\int_{J_{l}} \int_{B(0,R)}  \frac{|u(t,x)|^{1_{2}^{*}}}{|x|} \, dx dt & \geq  |J_{l}| \; \inf_{t \in J_l}  \textrm{Mass}^{1_{2}^{*}} \left( u(t), B(0,R) \right)
R^{\frac{2-3n}{n-2}} \cdot
\end{array}
\end{equation}
After summation over $l$ we see, by (\ref{Eqn:ConcMassZero})  and  (\ref{Eqn:MorawEst}) that

\begin{equation}
\begin{array}{l}
\sum\limits_{l=1}^{L} |J_{l}|  \left(  g^{- \gamma_{4}^{b}} (M)   |J_{l}|^{\frac{1}{2}} \right)^{1_{2}^{*}} \left( C_4
g^{\gamma_{4}^{a}} (M) |J_{l}|^{\frac{1}{2}} \right)^{\frac{2-3n}{n-2}} \\
\lesssim  |J|^{\frac{1}{2}},
\end{array}
\end{equation}
and after rearranging, we see that

\begin{equation}
\begin{array}{ll}
\sum \limits_{l=1}^{L} |J_{l}|^{\frac{1}{2}}
g^{- \left( \gamma_{4}^{b} 1_{2}^{*} + \gamma_{4}^{a} \frac{3n-2}{n-2} \right)}(M)   & \lesssim  |J|^{\frac{1}{2}} \cdot
\end{array}
\end{equation}

\item $J_{l}$ is exceptional. In this case by (\ref{Eqn:BoundCardExcep}) and

\begin{equation}
\begin{array}{ll}
\sum \limits_{l=1}^{L} |J_{l}|^{\frac{1}{2}} & \lesssim  \eta_{2}^{-1} \sup_{ 1 \leq
l \leq L} |J_{l}|^{\frac{1}{2}} \\
& \lesssim  \eta_{2}^{-1} |J|^{\frac{1}{2}} \cdot
\end{array}
\end{equation}

\end{itemize}

Therefore, writing $\sum\limits_{l=1}^{L} |J_{l}|^{\frac{1}{2}} \geq \frac{|J|}{\sup_{1 \leq l \leq L} |J_{l}|^{\frac{1}{2}} } $, we conclude that
there exists a positive constant $c_5 \ll 1$ and $\tilde{l} \in [1,..,L ]$ such that (\ref{Eqn:DistribInterv}) holds.

\end{proof}

\fbox{\textbf{Step 7}} \\

We use a crucial algorithm due to Bourgain \cite{bourg} to prove that there are many of those intervals that concentrate.

\begin{res}
Let $\eta := c_5 g^{-\gamma_5}(M) $. There exist a time $\bar{t}$, $K > 0$ and intervals $J_{l_{1}}$, ...., $J_{l_{K}}$ such that

\begin{equation}
\begin{array}{ll}
|J_{l_{1}}| \geq  2 |J_{l_{2}}| ... \geq 2^{k-1} |J_{l_{k}}| ... \geq 2^{K-1} |J_{l_{K}}|,
\end{array}
\end{equation}

\begin{equation}
\begin{array}{ll}
dist(\bar{t}, J_{l_{k}}  ) \leq \eta^{-1} |J_{l_k}|,
\end{array}
\label{Eqn:ControlDist}
\end{equation}
and

\begin{equation}
\begin{array}{ll}
K & \geq - \frac{\log{(L)}}{2 \log { \left(  \frac{\eta}{8} \right)}  } \cdot
\end{array}
\label{Eqn:LowerBoundK}
\end{equation}

\end{res}
A proof of this result in such a state can be found in \cite{triroyschrod} (see also \cite{taorad} from which the proof is inspired).
\\

\fbox{\textbf{Step 8}} \\

We prove that $L< \infty$, by using Step $7$.  More precisely

\begin{res}
There exists a constant  $C_{7} \gg 1$ such that

\begin{equation}
\begin{array}{l}
L \leq C_{7}^{ C_{7}g^{\gamma_7+} (M)},
\end{array}
\label{Eqn:BoundLFin}
\end{equation}
with $\gamma_7 := (2 \gamma_3 + \gamma_5) 1_{2}^{*} $.
\end{res}

\begin{proof}
Let $R_{l_k} := C g^{\gamma_3 + \gamma_5}(M) |J_{l_k}|^{\frac{1}{2}}$ with $C \gg 1$ a constant large enough such that all the statements below
are true. By Result \ref{res:ConcentrationMass} we have

\begin{equation}
\begin{array}{ll}
\textrm{Mass} \left( u(t),  B(x_{l_{k}},R_{l_k}) \right) & \geq c_3 g^{- \gamma_3} (M) |J_{l_k}|^{\frac{1}{2}}
\end{array}
\label{Eqn:LowerBondMass2}
\end{equation}
for all $t \in J_{l_{k}}$. By (\ref{Eqn:UpBdDerivM}) and (\ref{Eqn:ControlDist}) we see that (\ref{Eqn:LowerBondMass2}) holds for $t=\bar{t}$ with
$c_3$ replaced with $\frac{c_3}{2}$. On the other hand we see by (\ref{Eqn:MassControl}) that $($ Notation:
$\sum\limits_{k^{'}=k+N}^{K} a_{k^{'}}=0$, if $k+N > K$.  $)$

\begin{equation}
\begin{array}{ll}
\sum\limits_{k^{'}=k+N}^{K} \int_{B(x_{l_{k^{'}}} ,R_{l_k'} )} |u(\bar{t},x)|^{2} \, dx & \lesssim  \left( \frac{1}{2^{N}} + \frac{1}{2^{N+1}}.... +
\frac{1}{2^{K-k}} \right)  R_{l_k}^{2} \\
& \lesssim \frac{1}{2^{N-1}}  R_{l_k}^{2}
\end{array}
\end{equation}
Now we let  $N = C^{'} \log{(g(M))} $ with $C^{'} \gg  1 $ large enough  so that $ \frac{R_{l_k}^{2}}{2^{N-1}} \ll  c_3^{2}
g^{- 2 \gamma_3} (M) |J_{l_k}| $. By (\ref{Eqn:LowerBondMass2}) we have

\begin{equation}
\begin{array}{ll}
\sum\limits_{k^{'}=k+N}^{K} \int_{B(x_{_{l_{k^{'}}}} ,R_{l_k'} )} |u(\bar{t},x)|^{2} \, dx & \leq
\frac{1}{2}  \int_{B(x_{l_{k}},R_{l_k})} |u(\bar{t},x)|^{2} \, dx
\end{array}
\end{equation}
Therefore

\begin{equation}
\begin{array}{ll}
\int_{_{B(x_{l_{k}}, R_{l_k}) / \bigcup_{k^{'}=k+N}^{K} B (x_{l_{k^{'}}} ,R_{l_{k'}})  } } \   |u(\bar{t},x)|^{2} \, dx & \geq \frac{1}{2}  \int_{B(x_{l_{k}},R_{l_k})} |u(\bar{t},x)|^{2} \, dx \\
& \geq \frac{c_3^{2} g^{-2 \gamma_3} (M)}{4} |J_{l_k}|
\end{array}
\end{equation}
and by H\"older inequality, there exists a positive constant $\ll 1$ (that we still denote by $c_3$) such that

\begin{equation}
\begin{array}{ll}
\int_{_{B(x_{l_{k}}, R_{l_k}) / \bigcup_{k^{'}=k+N}^{K} B (x_{l_{k^{'}}} ,R_{l_{k'}})  } } \   |u(\bar{t},x)|^{1_{2}^{*}} \, dx & \geq c_3
g^{-  (2 \gamma_3 + \gamma_5)1_{2}^{*}} (M)
\end{array}
\end{equation}
and after summation over $k$, we have

\begin{equation}
\begin{array}{l}
\frac{K}{N} c_3 g^{- \left( 2 \gamma_3 + \gamma_5 \right) 1_{2}^{*}} (M) \lesssim 1,
\end{array}
\end{equation}
since $\sum\limits_{k=1}^{K} \chi_{ _{ B(x_{l_{k}}, R_{l_k}) / \cup_{k^{'}=k+N}^{K} B (x_{_{l_{k^{'}}}} ,R_{l_{k'}})}} \leq N$ and $\| u(t)
\|^{1_{2}^{*}}_{L^{1_{2}^{*}}} \lesssim  1$. Rearranging we see from (\ref{Eqn:LowerBoundK}) that there exists a constant $C_{7} \gg 1$  such that

\begin{equation}
\begin{array}{ll}
L & \leq  C_{7}^{ C_{7} g^{\gamma_7}(M)} \cdot
\end{array}
\label{Eqn:BoundL2}
\end{equation}
We see that (\ref{Eqn:BoundLFin}) holds. \\

\fbox{\textbf{Step 9}} \\

This is the final step. Recall that there are $L$ intervals $J_{l}$ and that on each of these intervals we have $\| u
\|_{L_{t}^{\bar{Q}} L_{x}^{\bar{R}} (J) }^{\bar{Q}} = \eta_{1}$. Therefore, there is a constant $ C_1 \gg 1$
such that (\ref{Eqn:BoundLong}) holds.

\end{proof}

\subsection{Proof of Lemma \ref{lem:regv}}
\label{Subsec:Lemmaregv}

In this subsection we prove Lemma \ref{lem:regv}. \\
\\
We write down some estimates.\\
\\
We have

\begin{equation}
\begin{array}{ll}
\| v_{1,h} - v_{1} \|_{L_{t}^{\infty} L_{x}^{1_{2}^{*}} ([t_{*} |J_{l}|, |J_{l}|])} & = \|  u_{l, (t_{*} - \eta_{3}) |J_{l}|,h} - u_{l,t_{1},h} -
(u_{l,(t_{*} - \eta_{3})|J_{l}|} - u_{l,t_{1}}  ) \|_{L_{t}^{\infty} L_{x}^{1_{2}^{*}}  ([t_{*} |J_{l}|, |J_{l}|])} \\
& \lesssim 1,
\end{array}
\label{Eqn:SobolevV1h}
\end{equation}
By the fundamental theorem of calculus (and the inequality $ \| D u \|_{L_{t}^{\infty} L_{x}^{2} ([ t_{*} |J_{l}|, |J_{l}| ] )} \lesssim
1$ ) we have

\begin{equation}
\begin{array}{ll}
\| u_{h} - u \|_{L_{t}^{\infty} L_{x}^{2} ([ t_{*} |J_{l}|, |J_{l}| ] )} &  \lesssim |h|
\end{array}
\label{Eqn:Diffu1}
\end{equation}
Moreover, by (\ref{Eqn:SobolevIneq1}) we have

\begin{equation}
\begin{array}{ll}
\| u_{h} -u \|_{ L_{t}^{\infty} L_{x}^{1_{2}^{*}} ([ t_{*} |J_{l}|, |J_{l}| ] )} \lesssim 1 \cdot
\end{array}
\label{Eqn:Diffu2}
\end{equation}
Let $\theta \in \left( 0, \frac{4}{n-2} \right)$. \\

There are two cases:

\begin{itemize}

\item $n \in \{ 3,4 \}$

Interpolating between (\ref{Eqn:Diffu1}) and (\ref{Eqn:Diffu2}) we get

\begin{equation}
\begin{array}{l}
\| u_{h}(s) - u(s) \|_{L^{\frac{n}{n-2}}}  \lesssim |h|^{\frac{n-2}{2}} \cdot
\end{array}
\nonumber
\end{equation}
The fundamental theorem of calculus, elementary estimates of $g$ and its derivatives,
(\ref{Eqn:EnergyBarely}), (\ref{Eqn:EquivF}), Proposition \ref{Prop:NonlinFracSmooth} combined with (\ref{Eqn:Sobklarge}), and Proposition \ref{Prop:Jensen}
$($ Observe from (\ref{Eqn:Interpkbar}) that Proposition \ref{Prop:Jensen} is applicable $)$ yield

\begin{equation}
\begin{array}{l}
\|  |u(s)|^{\frac{4}{n-2}} u(s) g(|u(s)|) - |u_{h}(s)|^{\frac{4}{n-2}} u_{h}(s) g(|u_{h}(s)|) \|_{L^{1}} \\
 \lesssim \| u_{h}(s)- u(s) \|_{L^{\frac{n}{n-2}}} \| u^{\frac{4}{n-2}}(s) g(|u(s)|) \|_{L^{\frac{n}{2}}} \\
\lesssim  \| u_{h}(s)- u(s) \|_{L^{\frac{n}{n-2}}} \left\|  u(s)^{\theta}  g^{\frac{\theta}{1_{2}^{*}}}(|u(s)|) \right\|_{L^{\frac{1_{2}^{*}}{\theta}}}
\left\| u(s)^{ \frac{4}{n-2} -\theta} g^{1 - \frac{\theta}{1_{2}^{*}}}(|u(s)|) \right\|_{L^{\frac{1_{2}^{*}}{\frac{4}{n-2}- \theta}}} \\
\lesssim |h|^{\frac{n-2}{2}} \left\| u(s) g^{\frac{1- \frac{\theta}{1_{2}^{*}}}{\frac{4}{n-2} - \theta }} (|u(s)|) \right\|^{\frac{4}{n-2} - \theta}_{L^{1_{2}^{*}}} \\
\lesssim g^{\frac{n-2}{n}+}(M) |h|^{\frac{n-2}{2}},
\end{array}
\nonumber
\end{equation}
by letting $\theta = \frac{4}{n-2}-$ at the last line. Hence by the dispersive inequality (\ref{Eqn:DispIneq}) we get

\begin{equation}
\begin{array}{ll}
\| v_{1,h}  - v_{1} \|_{L_{t}^{\infty} L_{x}^{\infty} ([t_{*}|J_{l}|, |J_{l}| ])} & \lesssim \eta_{3}^{\frac{2-n}{2}} |J_{l}|^{\frac{2-n}{2}}
g^{\frac{n-2}{n}+}(M) |h|^{\frac{n-2}{2}} \cdot
\end{array}
\nonumber
\end{equation}
Interpolating this inequality with (\ref{Eqn:SobolevV1h}) we get (\ref{Eqn:Regv}).\\

\item $n=5$

From (\ref{Eqn:Diffu1}) we get

\begin{equation}
\begin{array}{l}
\|  |u(s)|^{\frac{4}{3}} u(s) g(|u(s)|) - |u_{h}(s)|^{\frac{4}{3}} u_{h}(s) g(|u_{h}(s)|) \|_{L^{\frac{10}{9}}} \\
 \lesssim \| u_{h}(s)- u(s) \|_{L^{2}} \| u^{\frac{4}{3}}(s) g(|u(s)|) \|_{L^{\frac{5}{2}}} \\
\lesssim  \| u_{h}(s)- u(s) \|_{L^{2}}
\left\|  u(s)^{\theta}  g^{\frac{\theta}{1_{2}^{*}}}(|u(s)|) \right\|_{L^{\frac{1_{2}^{*}}{\theta}}}
\left\| u(s)^{ \frac{4}{3} -\theta} g^{1 - \frac{\theta}{1_{2}^{*}}}(|u(s)|) \right\|_{L^{\frac{1_{2}^{*}}{\frac{4}{3}- \theta}}} \\
\lesssim \left\| u g^{\frac{1- \frac{\theta}{1_{2}^{*}}}{\frac{4}{3} - \theta }} (|u|) \right\|^{\frac{4}{3} - \theta}_{L^{1_{2}^{*}}} \\
\lesssim g^{\frac{3}{5}+}(M) |h|,
\end{array}
\nonumber
\end{equation}
by letting $\theta = \frac{4}{n-2}-$ at the last line. Hence we get

\begin{equation}
\begin{array}{l}
\| v_{1,h}  - v_{1} \|_{L_{t}^{\infty} L_{x}^{10} ([t_{*}|J_{l}|, |J_{l}| ])} \lesssim
\eta_{3}^{-1} |J_l|^{-1} g^{\frac{3}{5}+}(M) |h| \cdot
\end{array}
\nonumber
\end{equation}
Interpolating this inequality with (\ref{Eqn:SobolevV1h}) we get (\ref{Eqn:Regv}).

\end{itemize}

\section{APPENDIX}
\label{Sec:Appendix}

In this appendix we prove Proposition \ref{Prop:LocalWell} by using a fixed point argument. \\
\\
First we prove an estimate in homogeneous Besov spaces (see e.g \cite{bahchem}) that will be used when we deal with the case $n=5$. Then we collect some estimates. Finally we write down the proof of Proposition \ref{Prop:LocalWell}. \\
\\
If $1 < k < 2$ let $\alpha \in (0,1)$ be such that $k= 1 + \alpha$. Let $\theta \in [0,1]$ and $C$ be two positive constants that are allowed to change from one line to the other one and such that all the estimates below are true. In addition $C$ is also allowed to change within the same line. The reader is urged to plot all the points
$ \left( \frac{1}{q}, \frac{1}{r} \right)$ wherever $L_{t}^{q} L_{x}^{r}$ appears on the coordinate plane $Oxy$ with $Ox$ (resp. $Oy$) representing the $x-$axis (resp. the $y-$ axis ).

\subsection{An estimate in homogeneous Besov spaces}

We prove the following lemma:

\begin{lem}
Assume that $ 0 < \alpha < 1$. Let $r$ and $\beta$ be such that $\alpha < \beta < 1$ and $r \beta \geq 1$. Let $H: \mathbb{R}^{2} \rightarrow \mathbb{R}^{2}$ be a  H\"older continuous function
with exponent $\beta$ which is $C^{1}$ (except at the origin) and which satisfies
$|H(f,\bar{f})| \approx |f|^{\beta}$ and $|H^{'}(f,\bar{f})| \approx |f|^{\beta -1}$. Let $  1 - \beta \gg \epsilon > 0$. Then

\begin{equation}
\begin{array}{ll}
\left\| H(f,\bar{f}) g(|f|) \right\|_{\dot{B}^{\alpha}_{r,r}} & \lesssim \| f \|^{\beta}_{\dot{B}^{\frac{\alpha}{\beta}}_{\beta r,\beta r}} +
\| f \|^{\beta + \epsilon}_{\dot{B}^{\frac{\alpha}{\beta +\epsilon}}_{(\beta + \epsilon)r, (\beta + \epsilon)r}}
\end{array}
\label{Eqn:BesovEst}
\end{equation}
\label{Lem:BesovEst}
\end{lem}

\begin{proof}

Let $ 1 > \mu > 0$ . Recall that if $ 0 < s < 1$ and $p \geq 1$ then

\begin{equation}
\begin{array}{l}
\| f \|^{p}_{\dot{B}^{s}_{p,p}} \approx \int_{\mathbb{R}^{n}} \frac{\left\| f(x+ h) - f(x) \right\|^{p}_{L^{p}}}{|h|^{n+  sp}} \; dh \cdot
\end{array}
\nonumber
\end{equation}
Elementary considerations (such as the estimate $g(|f|) \lesssim 1 + |f|^{\epsilon}$)   show that

\begin{itemize}

\item[$$]

\item if $|f|(x) \gg |f|(x+h)$ then

\begin{equation}
\begin{array}{l}
\left| H(f,\bar{f})g(|f|)(x+h) - H(f,\bar{f}) g(|f|)(x) \right|  \lesssim  \left| H(f,\bar{f}) g(|f|)(x) \right| \lesssim |f|^{\beta}(x)
+ |f|^{\beta + \epsilon}(x)  \\
\left| f (x+h) - f(x) \right|^{\mu} \gtrsim  \left| |f|^{\mu} (x+h) - |f|^{\mu}(x) \right|  \gtrsim |f|^{\mu}(x)
\end{array}
\nonumber
\end{equation}

\item[$$]

\item if $|f|(x+ h) \gg |f|(x)$ then the same estimates as above hold, except that $x$ (resp. $x+h$ ) is replaced with
$x+h$ (resp. $x$)

\item[$$]

\item if $|f|(x) \approx |f|(x+h)$ then there are two cases. If $|f|(x) \gg 1$ then

\begin{equation}
\begin{array}{l}
\left| H(f,\bar{f})g(|f|)(x+h) - H(f,\bar{f}) g(|f|)(x) \right|  \lesssim  \left( |f|^{\beta - 1}(x) + |f|^{\beta - 1 + \epsilon}(x) \right)
\left| f(x+ h) - f(x) \right|  \\
\left| f (x+h) - f(x) \right|^{\mu} \gtrsim  \left| |f|^{\mu}(x+h) - |f|^{\mu} (x) \right|  \gtrsim  |f|^{\mu-1}(x)  \left| f(x+ h) - f(x) \right| \cdot
\end{array}
\nonumber
\end{equation}
If $|f|(x) \lesssim 1$ then

\begin{equation}
\begin{array}{l}
\left| H(f,\bar{f})g(|f|)(x+h) - H(f,\bar{f}) g(|f|)(x) \right|  \lesssim  \left| f(x+ h) - f(x) \right|^{\beta} \cdot
\end{array}
\nonumber
\end{equation}

\end{itemize}
Hence dividing the region of integration of $\left\| H(f,\bar{f}) g(|f|) \right\|^{r}_{\dot{B}^{\alpha}_{r,r}}$
into the regions above, we see from the above estimates that (\ref{Eqn:BesovEst}) holds.

\end{proof}

\begin{rem}
A straightforward modification of the proof of Lemma \ref{Lem:BesovEst} shows that (\ref{Eqn:BesovEst}) also holds if
$g(|f|)$ is replaced with $g^{'}(|f|) |f|$ or $g^{''}(|f|) |f|^{2}$.
\end{rem}

\subsection{Some estimates}
\label{Subsec:Est}
We write down some basic estimates. They will be used in Subsection \ref{Subsec:Proof}. \\
Let $(\bar{q},\bar{r})$ be a bipoint. We have

\begin{equation}
\begin{array}{l}
1 \leq m \leq k: \; \| D^{m} u \|_{L_{t}^{\bar{q}}
L_{x}^{\bar{r}}([0,T_l])} \lesssim \| D^{k} u \|_{L_{t}^{\bar{q}} L_{x}^{\bar{r}}([0,T_l])}^{\theta}
\| D  u \|^{1- \theta}_{L_{t}^{\bar{q}} L_{x}^{\bar{r}}([0,T_l])} \cdot
\end{array}
\nonumber
\end{equation}
Let  $(\bar{q},\bar{r})$ be a bipoint such that $\bar{q} \geq \frac{2(n+2)}{n}$ and
$\frac{1}{\bar{q}} + \frac{n}{2 \bar{r}} = \frac{n}{4}$. Then

\begin{equation}
\begin{array}{l}
\| D^{\bar{k}} u \|_{L_{t}^{\bar{q}}
L_{x}^{\bar{r}}([0,T_l])} \lesssim \| D^{k} u \|_{L_{t}^{\frac{2(n+2)}{n}} L_{x}^{\frac{2(n+2)}{n}}([0,T_l])}^{\theta}
\| D  u \|^{1- \theta}_{L_{t}^{\infty} L_{x}^{2}([0,T_l])}  \cdot
\end{array}
\nonumber
\end{equation}
We have

\begin{equation}
\begin{array}{l}
\| D^{\alpha} u \|_{L_{t}^{\frac{2(n+2)}{n-2}} L_{x}^{\frac{2(n+2)}{n-2}}([0,T_l])}  \lesssim
\| D^{1+ \alpha} u \|_{L_{t}^{\frac{2n(n+2)}{n-2}} L_{x}^{\frac{2(n+2)}{n^{2}+ 4}} ([0,T_l])}, \; \text{and} \\
m \in \{1, 1+ \alpha \}: \; \| D^{m} u \|_{L_{t}^{\frac{2(n+2)}{n-2}} L_{x}^{\frac{2n(n+2)}{n^{2}+ 4}} ([0,T_l])}
\lesssim \| D^{m} u \|^{\theta}_{L_{t}^{\frac{2(n+2)}{n}} L_{x}^{\frac{2(n+2)}{n}}([0,T_l]) }
\| D^{m} u \|^{1 -\theta}_{L_{t}^{\infty} L_{x}^{2} ([0,T_l])} \cdot
\end{array}
\nonumber
\end{equation}
Let $\epsilon' := \frac{8(\bar{k}-1)}{_{(n-2)(n-2 \bar{k})}}$ and $0 < \epsilon^{''} < \epsilon^{'}$. Let $r$ be such that
$\frac{1}{\frac{n+2}{2} \left( \frac{4}{n-2} + \epsilon' \right)} = \frac{1}{r} - \frac{\bar{k}}{n}$. Observe that
$ \left( \frac{n+2}{2} \left( \frac{4}{n-2} + \epsilon^{'} \right), r \right)$ is admissible. We have

\begin{equation}
\begin{array}{l}
\| u \|_{L_{t}^{\frac{n+2}{2} \left( \frac{4}{n-2} + \epsilon'' \right)}
L_{r}^{\frac{n+2}{2} \left( \frac{4}{n-2} + \epsilon'' \right)} ([0,T_l]) } \lesssim
\| u \|^{\theta}_{L_{t}^{\frac{2(n+2)}{n-2}} L_{x}^{\frac{2(n+2)}{n-2}} ([0,T_l])}
\| u \|^{1- \theta}_{L_{t}^{\frac{n+2}{2} \left( \frac{4}{n-2} + \epsilon' \right)}
L_{x}^{\frac{n+2}{2} \left( \frac{4}{n-2} + \epsilon' \right)} ([0,T_l]) }, \; \text{and} \\
\| u \|_{L_{t}^{\frac{n+2}{2} \left( \frac{4}{n-2} + \epsilon' \right)}
L_{x}^{\frac{n+2}{2} \left( \frac{4}{n-2} + \epsilon' \right)} ([0,T_l]) } \lesssim \| D^{\bar{k}} u \|_{L_{t}^{\frac{n+2}{2} \left( \frac{4}{n-2} + \epsilon' \right)} L_{x}^{r}([0,T_l])} \cdot
\end{array}
\nonumber
\end{equation}
Let $\bar{\epsilon}^{'} := \frac{2(\bar{k}-1)(6-n)}{_{(n-2)(n-2\bar{k})}}$ and $0 < \bar{\epsilon}^{''} < \bar{\epsilon}^{'}$. Let
$r$ be such that $\frac{1}{\frac{2(n+2)}{6-n} \left( \frac{6-n}{n-2} + \bar{\epsilon}^{'} \right)} = \frac{1}{r} - \frac{\bar{k}}{n}$.
Observe that $ \left( \frac{2(n+2)}{6-n} \left( \frac{6-n}{n-2} + \bar{\epsilon}^{'} \right), r \right) $ is admissible. We have

\begin{equation}
\begin{array}{l}
\| u \|_{L_{t}^{\frac{2(n+2)}{6-n} \left( \frac{6-n}{n-2} + \bar{\epsilon}^{''} \right)}
L_{x}^{\frac{2(n+2)}{6-n} \left( \frac{6-n}{n-2} + \bar{\epsilon}^{''} \right) }([0,T_l]) }  \lesssim
\| u \|^{\theta}_{L_{t}^{\frac{2(n+2)}{n-2}} L_{x}^{\frac{2(n+2)}{n-2}} ([0,T_l]) }  \\
\hspace{6.7cm} \| u \|^{1- \theta}_{L_{t}^{\frac{2(n+2)}{6-n} \left( \frac{6-n}{n-2} + \bar{\epsilon}^{'} \right)}
L_{x}^{\frac{2(n+2)}{6-n} \left( \frac{6-n}{n-2} + \bar{\epsilon}^{'} \right) }([0,T_l])}, \; \text{and} \\
\| u \|_{L_{t}^{\frac{2(n+2)}{6-n} \left( \frac{6-n}{n-2} + \bar{\epsilon}^{'} \right)}
L_{x}^{\frac{2(n+2)}{6-n} \left( \frac{6-n}{n-2} + \bar{\epsilon}^{'} \right)} ([0,T_l])}  \lesssim
\| D^{\bar{k}} u \|_{L_{t}^{\frac{2(n+2)}{6-n} \left( \frac{6-n}{n-2} + \bar{\epsilon}^{'} \right)} L_{x}^{r}([0,T_l])} \; \cdot
\end{array}
\nonumber
\end{equation}
We assume that $n=3$ until the end of this subsection. Let $\tilde{\epsilon}^{'} := \frac{4 (\bar{k} -1)}{_{3 - 2 \bar{k}}} $ and $ 0 < \tilde{\epsilon}^{''} < \tilde{\epsilon}^{'}$. Let $r$ be such that
$\frac{1}{_{\frac{n+2}{4-n} \left( \frac{2(4-n)}{n-2} + \tilde{\epsilon}^{'} \right)}} = \frac{1}{r} - \frac{\bar{k}}{n}$.
Observe that $ \left( \frac{n+2}{4-n} \left( \frac{2(4-n)}{n-2} + \tilde{\epsilon}^{'} \right), r \right) $ is admissible. We have

\begin{equation}
\begin{array}{l}
\| u \|_{L_{t}^{\frac{n+2}{4-n} \left( \frac{2(4-n)}{n-2} + \tilde{\epsilon}^{''} \right)}
L_{x}^{\frac{n+2}{4-n} \left( \frac{2(4-n)}{n-2} + \tilde{\epsilon}^{''} \right) }([0,T_l]) }  \lesssim
\| u \|^{\theta}_{L_{t}^{\frac{2(n+2)}{n-2}} L_{x}^{\frac{2(n+2)}{n-2}} ([0,T_l]) }  \\
\hspace{6.7cm} \| u \|^{1- \theta}_{L_{t}^{\frac{n+2}{4-n} \left( \frac{2(4-n)}{n-2} + \tilde{\epsilon}^{'} \right)}
L_{x}^{\frac{n+2}{4-n} \left( \frac{2(4-n)}{n-2} + \tilde{\epsilon}^{'} \right) } ([0,T_l]) }, \; \text{and} \\
\| u \|_{L_{t}^{\frac{n+2}{4-n} \left( \frac{2(4-n)}{n-2} + \tilde{\epsilon}^{'} \right)}
L_{x}^{\frac{n+2}{4-n} \left( \frac{2(4-n)}{n-2} + \tilde{\epsilon}^{'} \right)} ([0,T_l])}  \lesssim
\| D^{\bar{k}} u \|_{L_{t}^{\frac{n+2}{4-n} \left( \frac{2(4-n)}{n-2} + \tilde{\epsilon}^{'} \right)} L_{x}^{r}([0,T_l])} \; \cdot
\end{array}
\nonumber
\end{equation}

\subsection{The proof}
\label{Subsec:Proof}

We define

\begin{equation}
\begin{array}{ll}
X & := \mathcal{C} ( [0,T_{l}], \tilde{H}^{k} ) \cap L_{t}^{\frac{2(n+2)}{n}}  D^{-1} L_{x}^{\frac{2(n+2)}{n}} ([0,T_{l}]) \cap
L_{t}^{\frac{2(n+2)}{n}}  D^{-k} L_{x}^{\frac{2(n+2)}{n}} ([0,T_{l}]) \cap L_{t}^{\frac{2(n+2)}{n-2}} L_{x}^{\frac{2(n+2)}{n-2}} ([0,T_{l}]) \cdot
\end{array}
\end{equation}
We also define for $C^{'} > 1 $ large enough  $($ here and in the sequel $\mathcal{B}(Z;r)$ denotes the closed ball centered at the origin with radius $r$ in the normed space $Z$. $)$

\begin{equation}
\begin{array}{ll}
X_{1} & := \mathcal{B} \left( \mathcal{C} ( [0,T_{l}], \tilde{H}^{k} ) \cap L_{t}^{\frac{2(n+2)}{n}} D^{-1}  L_{x}^{\frac{2(n+2)}{n}} ([0,T_{l}])
\cap L_{t}^{\frac{2(n+2)}{n}} D^{-k} L_{x}^{\frac{2(n+2)}{n}} ([0,T_l])
 ; C' M
\right)
\end{array}
\end{equation}
and, for $ 0 <  \delta:= \delta(M) \ll 1 $ small enough

\begin{equation}
\begin{array}{ll}
X_{2} & := \mathcal{B} \left( L_{t}^{\frac{2(n+2)}{n-2}} L_{x}^{\frac{2(n+2)}{n-2}} ([0,T_{l}]); 2 \delta \right) \cdot
\end{array}
\end{equation}
$X_{1} \cap X_{2} $ is a closed space of the Banach space $X$: therefore it is also a Banach space. Let

\begin{equation}
\begin{array}{l}
\Psi : = X_{1} \cap X_{2} \rightarrow X_{1} \cap X_{2} \\
u  \rightarrow e^{i t \triangle} u_0 - i \int_{0}^{t} e^{i(t-t^{'}) \triangle} \left(  |u|^{\frac{4}{n-2}} (t^{'}) u(t^{'}) g(|u|(t')) \right) \, d t^{'}
\end{array}
\end{equation}

\begin{itemize}

\item $\Psi$ maps $X_{1} \cap X_{2}$ to $X_{1} \cap X_{2}$. \\
\\
By the Strichartz estimates (\ref{Eqn:Strich}) and Corollary \ref{Cor:Nonlin} we have

\begin{equation}
\begin{array}{ll}
\| u \|_{L_{t}^{\infty} \tilde{H}^{k} ([0,T_{l}])} + \| D  u \|_{L_{t}^{\frac{2(n+2)}{n}} L_{x}^{\frac{2(n+2)}{n}} ([0,T_{l}])} + \| D^{k} u
\|_{L_{t}^{\frac{2(n+2)}{n}} L_{x}^{\frac{2(n+2)}{n}} ([0,T_{l}])} & \lesssim M + \delta^{\frac{4}{n-2}-} M g(M)
\end{array}
\label{Eqn:ExistLoc1}
\end{equation}
Moreover

\begin{equation}
\begin{array}{ll}
\| u \|_{L_{t}^{\frac{2(n+2)}{n-2}} L_{x}^{\frac{2(n+2)}{n-2}} ([0,T_{l}])} -  \| e^{ i t \triangle} u_{0} \|_{L_{t}^{\frac{2(n+2)}{n-2}}
L_{x}^{\frac{2(n+2)}{n-2}} ([0,T_{l}])} & \lesssim  \left\| D (|u|^{\frac{4}{n-2}} u g(|u|)  ) \right\|_{L_{t}^{\frac{2(n+2)}{n+4}}
L_{x}^{\frac{2(n+2)}{n+4}} ([0,T_{l}])} \\
& \lesssim   \delta^{\frac{4}{n-2}-} M g(M)
\end{array}
\end{equation}
so that

\begin{equation}
\begin{array}{ll}
\| u \|_{L_{t}^{\frac{2(n+2)}{n-2}} L_{x}^{\frac{2(n+2)}{n-2}} ([0,T_{l}])} -  \delta & \lesssim \delta^{\frac{4}{n-2}-} M g(M)
\end{array}
\label{Eqn:ExistLoc2}
\end{equation}
Therefore $\Psi(X_{1} \cap X_{2}) \subset X_{1} \cap X_{2}$. \\

\item $\Psi$ is a contraction. \\
\\
Given $\tau \in [0,1]$, let $w_{\tau} := u + \tau (v -u)$. Let $h(z,\bar{z}) := |z|^{\frac{4}{n-2}} z g(|z|)$. \\
By the fundamental theorem of calculus, the product rule (see proof of Proposition \ref{Prop:FracLeibn}), and the Sobolev
embedding (\ref{Eqn:SobolevIneq1}) we get

\begin{equation}
\begin{array}{l}
\| \Psi(u) - \Psi(v) \|_{X}  \lesssim  \sum\limits_{\beta \in \left\{1, k \right\}}
\left\| D^{\beta} \left( h(u,\bar{u}) - h(v,\bar{v}) \right) \right\|_{L_{t}^{\frac{2(n+2)}{n+4}} L_{x}^{\frac{2(n+2)}{n+4}}([0,T_l])} \\
\lesssim \sup_{\tau \in [0,1]}
\left[
\begin{array}{l}
\| u - v \|_{L_{t}^{\frac{2(n+2)}{n-2}} L_{x}^{\frac{2(n+2)}{n-2}}([0,T_l])}
 \sum\limits_{\substack{ \beta \in \left\{ 1, k \right\} \\ \tilde{z} \in \{ z,\bar{z} \}}} \| D^{\beta} \partial_{\tilde{z}} h(w_{\tau},\overline{w_{\tau}}) \|_{L_{t}^{\frac{n+2}{3}} L_{x}^{\frac{n+2}{3}} ([0,T_l])} \\
+  \sum\limits_{\substack{\beta \in \left\{1, k \right\} \\ \tilde{z} \in \{ z,\bar{z} \} }}
\| D^{\beta} (u-v) \|_{L_{t}^{\frac{2(n+2)}{n}} L_{x}^{\frac{2(n+2)}{n}}([0,T_l])}
\left\| \partial_{\tilde{z}} h(w_{\tau},\overline{w_{\tau}}) \right\|_{L_{t}^{\frac{n+2}{2}} L_{x}^{\frac{n+2}{2}}([0,T_l])}
\end{array}
\right]
\end{array}
\nonumber
\end{equation}
\\
From the estimates in Subsection \ref{Subsec:Est} and elementary estimates of $g$ (such as the estimate of $g$ that we use in the proof
of Lemma \ref{Lem:BesovEst}) we get

\begin{equation}
\begin{array}{ll}
\left\| \partial_{\tilde{z}} h(w_{\tau},\overline{w_{\tau}}) \right\|_{L_{t}^{\frac{n+2}{2}} L_{x}^{\frac{n+2}{2}}([0,T_l])}
& \lesssim M \left( \| w_{\tau} \|^{\frac{4}{n-2}}_{L_{t}^{\frac{2(n+2)}{n-2}} L_{x}^{\frac{2(n+2)}{n-2}} ([0,T_l])  }
+ \| w_{\tau} \|^{\frac{4}{n-2} + \epsilon''}_{L_{t}^{\frac{n+2}{2} \left( \frac{4}{n-2} + \epsilon'' \right) }
L_{x}^{\frac{n+2}{2} \left( \frac{4}{n-2} + \epsilon'' \right) }([0,T_l])} \right) \\
& \lesssim \langle M \rangle^{C} \delta^{C} \cdot
\end{array}
\nonumber
\end{equation}
We have

\begin{equation}
\begin{array}{ll}
\left\| D \partial_{\tilde{z}} h(w_{\tau},\overline{w_{\tau}}) \right\|_{L_{t}^{\frac{n+2}{3}} L_{x}^{\frac{n+2}{3}}([0,T_l])}
& \approx \left\| \nabla \partial_{\tilde{z}} h(w_{\tau},\overline{w_{\tau}})  \right\|_{L_{t}^{\frac{n+2}{3}} L_{x}^{\frac{n+2}{3}}([0,T_l])} \\
& \lesssim \| \nabla w_{\tau} \|_{L_{t}^{\frac{2(n+2)}{n}} L_{x}^{\frac{2(n+2)}{n}}([0,T_l])} \times Exp \cdot
\end{array}
\nonumber
\end{equation}
Here $Exp: =
\|  w_{\tau} \|^{\frac{6-n}{n-2}}_{L_{t}^{\frac{2(n+2)}{n-2}} L_{x}^{\frac{2(n+2)}{n-2}}([0,T_l]) }
+ \| w_{\tau} \|^{\frac{6-n}{n-2} + \bar{\epsilon}^{''}}_{L_{t}^{\frac{2(n+2)}{6-n} \left( \frac{6-n}{n-2} + \bar{\epsilon}^{''} \right)}
L_{x}^{\frac{2(n+2)}{6-n} \left( \frac{6-n}{n-2} + \bar{\epsilon}^{''} \right) }([0,T_l])}
$.
We have

\begin{equation}
\begin{array}{ll}
\left\| D^{1+ \alpha} \partial_{\tilde{z}} h(w_{\tau},\overline{w_{\tau}}) \right\|_{L_{t}^{\frac{n+2}{3}} L_{x}^{\frac{n+2}{3}}([0,T_l])}
& \lesssim \sum \limits_{\widetilde{w_{\tau}} \in  \{ w_{\tau}, \overline{w_{\tau}} \}} X_{1,\widetilde{w_{\tau}}} + X_{2,\widetilde{w_{\tau}}},
\end{array}
\nonumber
\end{equation}
with $ X_{1,\widetilde{w_{\tau}}} = \left\| D^{\alpha} \nabla  \left(  |w_{\tau}|^{\frac{4}{n-2}} \frac{\widetilde{w_{\tau}}}{_{\overline{w_{\tau}}}} g(|w_{\tau}|)  \right) \right\|_{L_{t}^{\frac{n+2}{3}} L_{x}^{\frac{n+2}{3}}([0,T_l])} $ and \\
$X_{2,\widetilde{w_{\tau}}} =  \left\|  D^{\alpha}  \nabla \left( |w_{\tau}|^{\frac{4}{n-2}} \frac{\widetilde{w_{\tau}}}{_{\overline{w_{\tau}}}} |w_{\tau}| g^{'}(|w_{\tau}|) \right) \right\|_{L_{t}^{\frac{n+2}{3}} L_{x}^{\frac{n+2}{3}}([0,T_l])} $. \\
We only estimate $X_{1,\widetilde{w_{\tau}}}$: $X_{2,\widetilde{w_{\tau}}}$ is estimated similarly. Expanding the gradient we see that we have to estimate terms of the form

\begin{equation}
\begin{array}{l}
Y := \left\| D^{\alpha} \left( \nabla w_{\tau} G(w_{\tau},\overline{w_{\tau}}) g(|w_{\tau}|) \right) \right\|_{L_{t}^{\frac{n+2}{3}} L_{x}^{\frac{n+2}{3}} ([0,T_l])}, \; \\
Z := \left\| D^{\alpha} \left( \nabla w_{\tau} G(w_{\tau},\overline{w_{\tau}}) g'(|w_{\tau}|) |w_{\tau}| \right)  \right\|_{L_{t}^{\frac{n+2}{3}} L_{x}^{\frac{n+2}{3}} ([0,T_l])}, \\
\bar{Z} := \left\| D^{\alpha} \left( \nabla w_{\tau} G(w_{\tau},\overline{w_{\tau}} ) g^{''}(|w_{\tau}|) |w_{\tau}|^{2} \right)  \right\|_{L_{t}^{\frac{n+2}{3}} L_{x}^{\frac{n+2}{3}} ([0,T_l])}, \; \text{and}
\end{array}
\nonumber
\end{equation}
terms that are similar to $Y$, $Z$, and $\bar{Z}$ (hence they are estimated similarly to $Y$, $Z$, and $\bar{Z}$). Here $G$ denotes a function of which the
regularity properties depend on the dimension. We only estimate $Y$: $Z$ and $\bar{Z}$ are estimated similarly. \\
\\
Assume that $n=3$. \\
Hence $G$ is a $\mathcal{C}^{1}$ function such that
$\tilde{G}(x) := G(x) g(|x|)$ satisfies  $ \left| \tilde{G}^{'}(x) \right| \lesssim |x|^{\frac{4}{n-2} -2} g(|x|)$ and
$\left| \tilde{G}^{'} \left( \tau x + (1- \tau) y \right) \right| \lesssim \left| \tilde{G}^{'}(x)  \right| + \left| \tilde{G}^{'}(y)  \right|  $
for all $\tau \in [0,1]$ $($ abuse of notation: $\tilde{G}(x)$, $G(x)$, and $G^{'}(x)$ denote $\tilde{G}(x,\bar{x})$, $G(x,\bar{x})$, and
$G^{'}(x,\bar{x})$, respectively $)$. From the fractional Leibnitz rule (see proof of Proposition \ref{Prop:FracLeibn}) we see that

\begin{equation}
\begin{array}{ll}
Y & \lesssim \| D^{\alpha} w_{\tau} \|_{L_{t}^{\frac{2(n+2)}{n-2}} L_{x}^{\frac{2(n+2)}{n-2}}([0,T_l])}
\| \nabla w_{\tau} \|_{L_{t}^{\frac{2(n+2)}{n}} L_{x}^{\frac{2(n+2)}{n}} ([0,T_l])} \\
&
\left(
\begin{array}{l}
\| w_{\tau} \|^{\frac{2(4-n)}{n-2}}_{L_{t}^{\frac{2(n+2)}{n-2}} L_{x}^{\frac{2(n+2)}{n-2}}([0,T_l])}  \\
+ \| w_{\tau} \|^{\frac{2(4-n)}{n-2} + \tilde{\epsilon}^{''}}_{L_{t}^{\frac{n+2}{4-n} \left( \frac{2(4-n)}{n-2} +
\tilde{\epsilon}^{''} \right)}
L_{x}^{\frac{n+2}{4-n} \left( \frac{2(4-n)}{n-2} + \tilde{\epsilon}^{''} \right)}([0,T_l])}
\end{array}
\right)  \\
& + \| D^{1 + \alpha} w_{\tau} \|_{L_{t}^{\frac{2(n+2)}{n}} L_{x}^{\frac{2(n+2)}{n}} ([0,T_l])} \; \times \; Exp \\
& \lesssim \langle M \rangle^{C} \delta^{C} \cdot
\end{array}
\nonumber
\end{equation}
Assume that $n=4$. \\
Observe that for this value of $n$ we have $\frac{2(n+2)}{n-2} = \frac{2(n+2)}{6-n}$. Let $r_1$, $r_2$, and $r_3$ be such that $ \frac{1}{r_1} = \frac{n}{2(n+2)} - \frac{\alpha}{n}$, $\frac{1}{r_2} = \frac{6-n}{2(n+2)} + \frac{\alpha}{n}$, and
$\frac{1}{r_3} = \frac{n-2}{2(n+2)} + \frac{1}{n} = \frac{n^{2}+4}{2n (n+2)} $.  Observe that $ \frac{1}{\frac{2(n+2)}{n-2}} + \frac{n}{2r_{3}} = \frac{n}{4} $.
The estimates below show that we may choose without loss of generality $ G(z,\bar{z})$ to be equal to $z$. First assume that $k<2$. We have

\begin{equation}
\begin{array}{ll}
Y & \lesssim  \| \nabla w_{\tau} \|_{L_{t}^{\frac{2(n+2)}{n}} L_{x}^{r_1} ([0,T_l])}
\left\| D^{\alpha}  \left( w_{\tau} g(|w_{\tau}|) \right) \right\|_{L_{t}^{\frac{2(n+2)}{6-n}} L_{x}^{r_2}([0,T_l])} \\
& + \| D^{1+ \alpha} w_{\tau} \|_{L_{t}^{\frac{2(n+2)}{n}} L_{x}^{\frac{2(n+2)}{n}} ([0,T_l])}
\| w_{\tau} g(|w_{\tau}|) \|_{L_{t}^{\frac{2(n+2)}{6-n}} L_{x}^{\frac{2(n+2)}{6-n}} ([0,T_l])} \cdot
\end{array}
\nonumber
\end{equation}
We have $ \| w_{\tau} g(|w_{\tau}|) \|_{L_{t}^{\frac{2(n+2)}{6-n}} L_{x}^{\frac{2(n+2)}{6-n}} ([0,T_l])}
 \lesssim Exp$. We have

\begin{equation}
\begin{array}{l}
\| \nabla w_{\tau} \|_{L_{t}^{\frac{2(n+2)}{n}} L_{x}^{r_1} ([0,T_l])}  \lesssim
\| D^{k} w_{\tau} \|_{L_{t}^{\frac{2(n+2)}{n}} L_{x}^{\frac{2(n+2)}{n}} ([0,T_l])}, \; \text{and} \\
\| w_{\tau} \|_{L_{t}^{\infty} L_{x}^{1_{2}^{*}} ([0,T_{l}])} \lesssim \| \nabla w_{\tau} \|_{L_{t}^{\infty} L_{x}^{2}([0,T_{l}])} \cdot
\end{array}
\nonumber
\end{equation}
We also have

\begin{equation}
\begin{array}{l}
\left\| D^{\alpha} \left( w_{\tau} g(|w_{\tau}|) \right) \right\|_{L_{t}^{\frac{2(n+2)}{6-n}} L_{x}^{r_2} ([0,T_l])}
\lesssim \left\| \nabla \left( w_{\tau} g(|w_{\tau}|) \right) \right\|^{\theta}_{L_{t}^{\frac{2(n+2)}{6-n}} L_{x}^{r_3} ([0,T_l])}
\times  Exp^{1- \theta} ,
\end{array}
\nonumber
\end{equation}
and, from elementary estimates of $g$ (such as $g(|f|) + g^{'}(|f|) |f| \lesssim 1 + |f|^{\epsilon}$ for $ \epsilon > 0 $ ),

\begin{equation}
\begin{array}{ll}
\left\| \nabla \left( w_{\tau} g(|w_{\tau}|) \right)  \right\|_{L_{t}^{\frac{2(n+2)}{6-n}} L_{x}^{r_3} ([0,T_l])} & \lesssim
\left\| \nabla w_{\tau}  \right\|_{L_{t}^{\frac{2(n+2)}{6-n}} L_{x}^{r_3}([0,T_l])} +
\| \nabla w_{\tau} \|_{L_{t}^{\frac{2(n+2)}{6-n}} L_{x}^{r_3^{+}}([0,T_l])}
\| w_{\tau} \|^{C}_{L_{t}^{\infty} L_{x}^{1^{*}_{2}} ([0,T_l])} \cdot
\end{array}
\nonumber
\end{equation}
We also have

\begin{equation}
\begin{array}{ll}
\| \nabla w_{\tau} \|_{L_{t}^{\frac{2(n+2)}{6-n}} L_{x}^{r_3^{+}}([0,T_l])} & \lesssim
\| D^{1+} w_{\tau} \|_{L_{t}^{\frac{2(n+2)}{n-2}} L_{x}^{\frac{2n (n+2)}{n^{2} +4}} ([0,T_l])} \cdot
\end{array}
\nonumber
\end{equation}
Hence $ Y \lesssim \langle M \rangle^{C} \delta^{C} $. \\
Then assume that $k=2$. This case boils down to the previous one by replacing $\alpha$ with $1$ in all the estimates. \\
\\
Assume now that $n=5$. \\
Hence $G$ is an H\"older continuous function with exponent $\frac{1}{3}$ that is $\mathcal{C}^{1}$ (except at the origin) that satisfies
$|G(f,\bar{f})| \approx |f|^{\frac{1}{3}}$ and $|G^{'}(f,\bar{f})| \approx |f|^{-\frac{1}{3}}$. We only estimate $Y$: $Z$ is estimated similarly. Let $r$ and $\tilde{r}$ be such that $ \frac{1}{\tilde{r}} = \frac{n}{2(n+2)} - \frac{\alpha}{n} $ and $\frac{1}{\tilde{r}} + \frac{1}{r} = \frac{3}{n+2}$. We have $ Y \lesssim Y_1 + Y_{2} $ with

\begin{equation}
\begin{array}{l}
Y_{1} := \bar{Y}_{1} \| \nabla w_{\tau} \|_{L_{t}^{\frac{2(n+2)}{n}} L_{x}^{\tilde{r}-}([0,T_l])}, \; \text{and}  \\
Y_{2} := \left\| G(w_{\tau},\overline{w_{\tau}}) g(|w_{\tau}|) \right\|_{L_{t}^{\frac{2(n+2)}{6-n}} L_{x}^{\frac{2(n+2)}{6-n}} ([0,T_l])}
\| D^{1+ \alpha} w_{\tau} \|_{L_{t}^{\frac{2(n+2)}{n}} L_{x}^{\frac{2(n+2)}{n}} ([0,T_l])},
\end{array}
\nonumber
\end{equation}
with $ \bar{Y}_{1} := \left\| D^{\alpha} \left( G(w_{\tau},\overline{w_{\tau}}) g(|w_{\tau}|) \right) \right\|_{L_{t}^{\frac{2(n+2)}{6-n}} L_{x}^{r+} ([0,T_l])} $.
We have

\begin{equation}
\begin{array}{ll}
\| \nabla w_{\tau} \|_{L_{t}^{\frac{2(n+2)}{n}} L_{x}^{\tilde{r}-}([0,T_l])} & \lesssim
\| D^{k-} w_{\tau} \|_{L_{t}^{\frac{2(n+2)}{n}} L_{x}^{\frac{2(n+2)}{n}} ([0,T_{l}])} \\
& \lesssim \| D w_{\tau} \|^{\theta}_{L_{t}^{\frac{2(n+2)}{n}} L_{x}^{\frac{2(n+2)}{n}} ([0,T_{l}])}
\| D^{k} w_{\tau} \|^{1 - \theta}_{L_{t}^{\frac{2(n+2)}{n}} L_{x}^{\frac{2(n+2)}{n}} ([0,T_{l}])} \cdot
\end{array}
\nonumber
\end{equation}
Hence $ \| \nabla w_{\tau} \|_{L_{t}^{\frac{2(n+2)}{n}} L_{x}^{\tilde{r}-}([0,T_l])} \lesssim M $. We have

\begin{equation}
\begin{array}{ll}
Y_{2} & \lesssim \left( \| w_{\tau} \|^{\frac{6-n}{n-2}}_{L_{t}^{\frac{2(n+2)}{n-2}} L_{x}^{\frac{2(n+2)}{n-2}} ([0,T_l])} +
\| w_{\tau} \|^{\frac{6-n}{n-2} + \bar{\epsilon}^{''}}_{L_{t}^{\frac{2(n+2)}{6-n} \left( \frac{6-n}{n-2} + \bar{\epsilon}^{''} \right)}
L_{x}^{\frac{2(n+2)}{6-n} \left( \frac{6-n}{n-2} + \bar{\epsilon}^{''} \right)} ([0,T_l])} \right)  \\
& \| D^{1+ \alpha} w_{\tau} \|_{L_{t}^{\frac{2(n+2)}{n}} L_{x}^{\frac{2(n+2)}{2}} ([0,T_l])} \\
& \lesssim \langle M \rangle^{C} \delta^{C} \cdot
\end{array}
\nonumber
\end{equation}
We then estimate $\bar{Y}_1$. In the sequel, if $x \in \mathbb{R}$ then, in order to simplify the notation, we allow the value of $x+$ (resp. $x-$) to change from one line to the other one and even within the same line; the values of $x^{+}$ and those of $x^{-}$ are chosen such that all the estimates below are true. From the definition of Besov spaces using the Paley-Littlewood projectors, a Paley-Littlewood decomposition into low frequencies and high frequencies,  and elementary estimates such as Bernstein inequalities and H\"older inequality for sequences $\bar{Y}_{1} \lesssim \bar{Y}_{1,a} + \bar{Y}_{1,b} $ with
$\bar{Y}_{1,a} :=  \left\|  G(w_{\tau},\overline{w_{\tau}})  g(|w_{\tau}|) \right\|_{L_{t}^{\frac{2(n+2)}{6-n}} \dot{B}^{(\alpha-)}_{r+,r+} ([0,T_l])} $ and
$\bar{Y}_{1,b} :=  \left\| G(w_{\tau},\overline{w_{\tau}}) g(|w_{\tau}|) \right\|_{L_{t}^{\frac{2(n+2)}{6-n}} \dot{B}^{(\alpha+)}_{r+,r+} ([0,T_l])}$.
\\
\\
We first estimate $\bar{Y}_{1,b}$. From Lemma \ref{Lem:BesovEst} we get

\begin{equation}
\begin{array}{ll}
\bar{Y}_{1,b}   \lesssim \| w_{\tau} \|^{\frac{1}{3}}_{L_{t}^{\frac{2(n+2)}{n-2}} \dot{B}^{(3 \alpha)+}_{\frac{r}{3}+, \frac{r}{3}+} ([0,T_{l}])}
+ \| w_{\tau} \|^{\frac{1}{3}+}_{L_{t}^{\frac{2(n+2)}{n-2}+} \dot{B}^{(3 \alpha)+}_{\frac{r}{3}+, \frac{r}{3}+}([0,T_{l}])} \cdot
\end{array}
\nonumber
\end{equation}
Hence by using again the definition of Besov spaces using the Paley-Littlewood projectors $\bar{Y}_{1,b}$ is bounded by powers of terms of the form
$ A := \left\| D^{(3 \alpha)+} w_{\tau} \right\|_{L_{t}^{\frac{2(n+2)}{n-2}}  L_{x}^{\frac{r}{3}+} ([0,T_{l}])} $
and $ B := \left\| D^{(3 \alpha)+} w_{\tau} \right\|_{L_{t}^{\frac{2(n+2)}{n-2}+} L_{x}^{\frac{r}{3}+} ([0,T_{l}]) } $. We have

\begin{equation}
\begin{array}{ll}
A & \lesssim  \| w_{\tau} \|^{\theta}_{L_{t}^{\frac{2(n+2)}{n-2}} L_{x}^{\frac{2(n+2)}{n-2}} ([0,T_{l}]) }
\| D^{1+} w_{\tau} \|^{ 1- \theta}_{L_{t}^{\frac{2(n+2)}{n-2}} L_{x}^{\frac{2n(n+2)}{n^{2}+4}} ([0,T_{l}]) }, \; \text{and} \\
B & \lesssim  \| w_{\tau} \|^{\theta}_{L_{t}^{\frac{2(n+2)}{n-2}+} L_{x}^{\frac{2(n+2)}{n-2}-} ([0,T_{l}]) }
\| D^{1+} w_{\tau} \|^{ 1- \theta}_{L_{t}^{\frac{2(n+2)}{n-2}+} L_{x}^{\frac{2n(n+2)}{n^{2}+4}-} ([0,T_{l}]) } \cdot
\end{array}
\label{Eqn:BoundAB}
\end{equation}
If $(q_{1}, r_{1}) :=  \left( \frac{2(n+2)}{n-2}, \frac{2n(n+2)}{n^{2}+4} \right) $
or $(q_{1}, r_{1}) :=  \left( \frac{2(n+2)}{n-2}+, \frac{2n(n+2)}{n^{2}+4}- \right) $ then

\begin{equation}
\begin{array}{ll}
\| D^{1+} w_{\tau} \|_{L_{t}^{q_{1}} L_{x}^{r_{1}} ([0,T_{l}])} & \lesssim
\| D^{1+} w_{\tau} \|^{\theta}_{L_{t}^{\frac{2(n+2)}{n}} L_{x}^{\frac{2(n+2)}{n}} ([0,T_{l}])} \| D^{1+} w_{\tau} \|^{1- \theta}_{L_{t}^{\infty} L_{x}^{2} ([0,T_{l}])} \cdot
\end{array}
\nonumber
\end{equation}
We also have

\begin{equation}
\begin{array}{ll}
\| w_{\tau} \|_{L_{t}^{\frac{2(n+2)}{n-2}+} L_{x}^{\frac{2(n+2)}{n-2}-} ([0,T_{l}]) } & \lesssim
\| w_{\tau} \|^{\theta}_{L_{t}^{\frac{2(n+2)}{n-2}} L_{x}^{\frac{2(n+2)}{n-2}} ([0,T_{l}]) } \| w_{\tau} \|^{1- \theta}_{L_{t}^{\infty} L_{x}^{1_{2}^{*}} ([]0,T_{l})} \\
& \lesssim \| w_{\tau} \|^{\theta}_{L_{t}^{\frac{2(n+2)}{n-2}} L_{x}^{\frac{2(n+2)}{n-2}} ([0,T_{l}]) } \| D w_{\tau} \|^{1- \theta}_{L_{t}^{\infty} L_{x}^{2} ([0,T_{l}])} \cdot
\end{array}
\nonumber
\end{equation}
Hence $\bar{Y}_{1,b} \lesssim  \langle M \rangle^{C} \delta^{C}$. \\ We then estimate $\bar{Y}_{1,a}$ in a similar fashion. We have

\begin{equation}
\begin{array}{ll}
\bar{Y}_{1,a} &  \lesssim \| w_{\tau} \|^{\frac{1}{3}}_{L_{t}^{\frac{2(n+2)}{n-2}} \dot{B}^{(3 \alpha)-}_{\frac{r}{3}+, \frac{r}{3}+} ([0,T_{l}])}
+ \| w_{\tau} \|^{\frac{1}{3}+}_{L_{t}^{\frac{2(n+2)}{n-2}+} \dot{B}^{(3 \alpha)-}_{\frac{r}{3}+, \frac{r}{3}+}([0,T_{l}])} \cdot
\end{array}
\nonumber
\end{equation}
Hence $\bar{Y}_{1,a}$ is bounded by powers of terms of the form
$ A := \left\| D^{(3 \alpha)-} w_{\tau} \right\|_{L_{t}^{\frac{2(n+2)}{n-2}}  L_{x}^{\frac{r}{3}+} ([0,T_{l}])} $
and $ B := \left\| D^{(3 \alpha)-} w_{\tau} \right\|_{L_{t}^{\frac{2(n+2)}{n-2}+} L_{x}^{\frac{r}{3}+} ([0,T_{l}]) } $. Since $A$ and $B$ are bounded by
the same bounds that appear in (\ref{Eqn:BoundAB}), we get $ \bar{Y}_{1,a} \lesssim  \langle M \rangle^{C} \delta^{C} $.

\end{itemize}

\vspace{1cm}

\textbf{Funding}: Part of this work was done in Japan. The work in
Japan was supported by a JSPS Kakenhi [ 15K17570 to T.R. ] \\
\\
\textbf{Acknowledgments}: The author would like to thank Nobu Kishimoto for discussions related to this problem.

\end{document}